\documentclass[MNA]{mna} %

\usepackage{bbm}
\usepackage{enumerate}
\usepackage{stmaryrd}

\theoremstyle{mnabodyrm}
\newtheorem{modstep}[theorem]{Modelization Step}

\makeatletter
\newcommand{\mylabel}[2]{#2\def\@currentlabel{#2}\label{#1}}
\makeatother

\SHORTTITLE{On the long time behaviour of  stochastic Hodgkin-Huxley neurons}

\TITLE{On the long time behaviour of 
	single stochastic Hodgkin-Huxley neurons with constant signal, 
	and a construction of circuits of interacting neurons 
	showing self-organized rhythmic oscillations } 

\AUTHORS{%
  Reinhard~H\"opfner\footnote{Inst. f\"ur Mathematik, Johannes Gutenberg Univ. 
  Mainz, \EMAIL{hoepfner@mathematik.uni-mainz.de}}}

\SHORTAUTHOR{Reinhard~H\"opfner}


\AMSSUBJ{60J25} 
\AMSSUBJSECONDARY{60F15, 60K35, 62M99} 

\SUBMITTED{March 31, 2022} 
\ACCEPTED{December 7, 2022} 

\ARXIVID{2203.16160} 

\VOLUME{3}
\YEAR{2023}
\PAPERNUM{1}
\DOI{10.46298/mna.9279}

\ABSTRACT{ The stochastic Hodgkin-Huxley neurons considered in this paper replace  
time-constant deterministic input $a dt$ of the classical deterministic model by increments 
$\vartheta dt + dX_t$ of a stochastic process:  $X$ is Ornstein-Uhlen\-beck with volatility 
$\sigma>0$ and back-driving force $\tau>0$, and we call $\vartheta>0$ the signal.  
	We have ergodicity and strong laws of large numbers for various functionals of the process, 
	and characterize 'quiet behaviour' and 'regular spiking' as events whose probability 
	depends on the parameters $(\tau,\sigma)$ and on the signal $\vartheta$. 
	The notions of quiet behaviour and regular spiking allow for a construction of circuits of 
	interacting stochastic Hodgkin-Huxley neurons, combining excitation with inhibition 
	according to a block structure along the circuit, on which self-organized rhythmic 
	oscillations 
	can be observed. .}

\begin{document}



\section{Introduction}
Self-organized rhythmic oscillation in stochastic systems has been studied in different 
contexts in biology and physics. Cerf, Dai Pra,  Formentin and Tovazzi \cite{CDFT-21} study 
spin systems with nearest neighbour interaction along a circuit which show the following 
behaviour. Starting from magnetisation (all spins equal to $1$, say) a rather long waiting time 
is needed to observe flipping of a first spin, rapidly followed --in virtue of the structure of the 
interaction-- by spins flipping at successive neighbouring sites along the circuit which leads 
to  magnetisation of opposite sign (all spins equal to $-1$, say). Then again, with roles of signs 
interchanged, a rather long time is needed to observe a first spin flipping back, rapidly 
followed by successive neighbours, and the circuit returns to its initial state of magnetization. 
This creates a self-organized rhythmic oscillation in a Markovian system which is 
homogeneous in time. The authors can prove that this oscillation is  persistent.

Ditlevsen and L\"ocherbach~\cite{DL-17} study circuits of blocks of neurons --where 
interaction 
between successive blocks is of mean-field type-- where neurons are modelled through 
Hawkes processes, i.e. Poissonian point processes where intensity is a function of past 
spiking activity in preceding blocks. Non-Markovian in general, specific memory kernels 
however allow an expansion of the structure into Markovian cascades, \textit{i.e.} finite 
sequences of 
successive Markovian steps corresponding to every block. Now Markovian ergodicity tools are 
again at hand, together with mean field limits in large blocks.  When inhibition and excitation is 
properly balanced, the authors prove that in the limit the circuit behaves as a deterministic 
system enjoying the following properties (Theorem~3 in Section~4 of \cite{DL-17}): i) exactly 
one equilibrium point exists for the system; ii) this equilibrium point is unstable, iii) there is a 
stable periodic orbit for the system; iv) other periodic orbits may exist, but at most in finite 
number.

The aim of the present paper is to show that similar patterns of self-organized rhythmic 
oscillation can be observed in the spike trains of certain circuits of interacting stochastic 
Hodgkin-Huxley neurons, under suitable balance of excitation and inhibition according to a 
block structure in the circuit, and under careful determination of suitable 'levels of noise'.

In this view, a first ingredient is to work out, for single stochastic Hodgkin-Huxley neurons 
receiving input  $\vartheta dt + dX_t$ where $X$ is an Ornstein-Uhlenbeck process with 
back-driving force $\tau>0$ and volatility $\sigma>0$, a notion of 'quiet 
behaviour' and a 
notion of 'regular spiking' such that, for suitable pairs $(\tau,\sigma)$ characterizing the level 
of 'noise' and with large probability, 'regular spiking' will be observed for suitably large values 
$\vartheta_2$ of the signal, and 'quiet behaviour'  for suitably small values $\vartheta_1$.  
Quiet behaviour on a time interval of certain length will be defined through a comparison with 
Poisson processes of very low intensity, and regular spiking in terms of quantiles of interspike 
times clustering around their median. 
Simulations provide evidence that it is not sufficient to choose the signal alone strong enough 
(\textit{e.g.}, such that 
trajectories of a deterministic Hodgkin-Huxley neuron with constant input would be attracted 
by a stable orbit, for almost all initial conditions, in presence of an unstable equilibrium point), 
or  the signal alone weak enough (\textit{e.g.}, such that trajectories of a deterministic 
Hodgkin-Huxley neuron with constant input would be attracted by a stable fixed point, for 
almost all initial conditions): the essential condition in stochastic neurons is an interplay, in 
dependence on suitable pairs of values $\vartheta_1 < \vartheta_2$ for the signal, between 
volatility (sufficiently  small) and back-driving force (sufficiently strong).

The second ingredient is to associate to every neuron in the circuit an output process, solution 
to an Ornstein-Uhlenbeck type SDE driven by the point process of its spikes, with positive 
back-driving force. We have to transform this output into input for a successor neuron. In a 
circuit of $N=ML$ neurons, ordered in $M$ blocks containing $L$ neurons each, and where 
we 
count neurons modulo $N$ around the circuit, the output of neuron $i{-}1$ transforms into 
input for neuron $i$ in the following way: neurons which occupy first positions in their 
respective blocks receive bounded inhibitory input, neurons having their predecessor in the 
same block (\textit{i.e.} all others) receive bounded excitatory input. With suitable choice of 
bounded 
monotone functions $\Phi_{\mathrm{inh}}$ (decreasing) and  $\Phi_{\mathrm{exc}}$ 
(increasing), writing 
$U^{(j)}$ for the output produced by neuron $j$, the input which neuron $i$ receives is thus 
\[
\Phi_{\mathrm{inh}}(U^{(i-1)}) \quad\mbox{ if } i = 1 \quad\mbox{\textrm{modulo}}\quad L  , 
\quad 
\Phi_{\rm 
exc}(U^{(i-1)}) \quad\mbox{else}   
\]
(in particular, neuron $1$ receives input $\Phi_{\mathrm{inh}}(U^{(N)})$, counting modulo 
$N$ 
around the circuit). 
Suitably balanced and under the condition that the number $M$ of blocks is odd, simulations 
make appear oscillating patterns of spiking activity around the circuit in the sense that blocks 
of neurons flip from regular spiking regime into quiet regime and from quiet regime into 
regular spiking regime. This creates a slow rhythmic oscillation of activity patterns around the 
circuit which seems persistent. Simulation results as represented in 
Figures~\ref{fig:circuit_UJ0=0} and \ref{fig:circuit_Uj0random} illustrate this phenomenon, 
already for 
small values of $L$ and $M$, and show that rhythmic oscillation establishes itself rather 
rapidly.

We have no proof that the observed slow rhythmic oscillation of spiking activity around the 
circuit is indeed persistent. A heuristic argument however 
might be as follows.  Think of a deterministic system of dimension $N=LM$ where variables 
$t \to x_j(t)$ represent in some way a spiking activity of neuron $j$ as a function of time, 
and where counting modulo $N$ the interaction scheme is of type   
\[
\dfrac{d x_i}{dt}(t) = 
\begin{cases}
	-c x_i(t) - f(x_{i-1}(t))  &\quad\mbox{ if } \quad i  = 1 \quad \mbox{ \textrm{modulo}} \quad 
	L     \\
	-c x_i(t) + f(x_{i-1}(t))   &\quad\mbox{else}  
\end{cases}
\]
with $f$ some smooth function which is close to the truncation function $x \to (x\vee 
{-}1)\wedge 1$, and with $c>0$ some constant. Under the condition that i) $M$ is odd and ii) 
$c$ is small enough, this system evolves on a finite number of periodic orbits, and at least one 
periodic orbit is stable. This is again Theorem 3 in section 4 of \cite{DL-17}, the system 
$t \to (x_1(t),\ldots,x_N(t))$ being a simplified version of the deterministic limit system 
considered 
there. In simulations under random initial conditions, the slow rhythmic oscillation of activity 
patterns in the circuit of stochastic Hodgkin-Huxley neurons constructed above looks very 
much like those in the deterministic system $t \to (x_1(t),\ldots,x_N(t))$.

The main effort of the present paper is on modelization and balance, a key ingredient being a 
rigorous definition of notions such as quiet behaviour and regular spiking in stochastic 
Hodgkin-Huxley 
neurons with constant signal. Proofs that the oscillating behaviour observed in finite circuits of 
stochastic Hodgkin-Huxley neurons is indeed persistent (certainly perturbed by randomness 
from time to time but always re-establishing itself rather rapidly) remains an open and  
challenging problem.

The present paper is organized as follows.  
At the core of the paper, Section~\ref{sect:circuits} is devoted to the construction of circuits 
which exhibit self-organized oscillation. This section does not contain proofs. As a preparation 
for Section ~\ref{sect:circuits}, all other sections except the first one  (which recalls some 
known facts for classical deterministic Hodgkin-Huxley neurons with constant input) focus on 
the single stochastic Hodgkin-Huxley neuron as a Harris recurrent strong Markov process. 

In particular, Section~\ref{sect:stoch_HH} introduces the stochastic Hodgkin-Huxley neuron 
with constant signal, sketches its  ergodicity properties and states some strong laws of large 
numbers, in particular for empirical distribution functions of spiking patterns. Proofs based on 
artificially defined life cycles through Nummelin splitting (methods as in H\"opfner, 
L\"ocherbach and Thieullen \cite{HLT-16, HLT-16a, HLT-17}) are collected  in an 
appendix Section~\ref{sect:proof_prop_2.5}. 
The process of 'output' of a stochastic neuron, key tool in view of modelization of interactions 
along circuits, 
is defined in Section~\ref{output}. 
Quantifying a comparison with Poisson processes of very low intensity, 
Section~\ref{sect:quiet} defines quiet behaviour of a single stochastic Hodgkin-Huxley neuron 
with 
constant signal as an event whose probability depends on the noise level and the value of the 
signal. 
Section~\ref{sect:regular_spiking} defines regular spiking in terms of quantiles of interspike 
times which cluster sufficiently close to  their median. Consequences for the output process 
(based on two conjectures --which we believe realistic-- on concentration properties of the 
limit of empirical distribution function for interspike times) are discussed in an appendix 
Section~\ref{conjectures}.

The limit theorems of Section~\ref{sect:stoch_HH}  and the notions in 
Sections~\ref{sect:quiet}--\ref{sect:regular_spiking} form the basis  for the construction of 
circuits of 
interacting stochastic Hodgkin-Huxley neurons in Section~\ref{sect:circuits}, whereas 
Appendices~\ref{sect:proof_prop_2.5}--\ref{conjectures} may be left for further reading. 

\texttt{R} code underlying our simulations in the present paper is provided under 
\url{http://modeldb.yale.edu/267611}.

\section{Deterministic Hodgkin-Huxley model with constant rate of 
input}\label{sect:det_HH}

Hodgkin-Huxley models  \cite{HH-52} play an important role in 
neuroscience and are considered as realistic models for the spiking behaviour of neurons. For 
an overview see Izhikevich~\cite{Izh-07} and Ermentrout and Terman~\cite{ET-10}. 
The classical deterministic model with constant rate of input is a $4$-dimensional dynamical 
system with variables $(V,n,m,h)$  
\begin{equation}\label{det_HH}
\begin{cases}
		dV_t =  a dt - F(V_t,n_t,m_t,h_t) dt \\
		dn_t =  [ \alpha_n(V_t) (1-n_t) - \beta_n(V_t) n_t ] dt \\
		dm_t =  [ \alpha_m(V_t) (1-m_t) - \beta_m(V_t) m_t ] dt \\
		dh_t =  [ \alpha_h(V_t) (1-h_t) - \beta_h(V_t) h_t ] dt 
\end{cases}
\end{equation}
where $a>0$ is a constant.  We define the functions $F$ and $\alpha_j$, $\beta_j$, 
$j\in\{n,m,h\}$, as in Izhikevich \cite[pp. 37--38]{Izh-07}  (different choices for the constants 
exist in the literature):  
\begin{equation}\label{eq:F}
	F(v,n,m,h) := 36 n^4 (v+12) +120 m^3 h (v -120) + 0.3(v-10.6), 
\end{equation}
\begin{align}
	\nonumber
		\alpha_n(v)  &= \frac{0.1-0.01v }{\exp(1-0.1v)-1} &\beta_n(v)  &=  0.125\exp(-v/80) 
		,  \\\label{eq:alphabeta}
		\alpha_m(v) &= \frac{2.5-0.1v}{\exp(2.5-0.1v)-1} &\beta_m(v) &=  4\exp(-v/18) ,  \\
		\alpha_h(v) &= 0.07\exp(-v/20)  &\beta_h(v) &= \frac{1}{\exp(3-0.1v)+1} . 
	\nonumber
\end{align}

The variable $V$ takes values in $\mathbb{R}$ and models the membrane potential in the 
single neuron. The variables $n$, $m$, $h$ are termed gating variables (or internal variables) 
and take values in $[0,1]$. The state space for this system is  
$E_4:=\mathbb{R}\times[0,1]^3$. In the sequel, for reasons which will appear in 
Section~\ref{sect:stoch_HH}, we  shall speak of  $a>0$ in \eqref{det_HH} as a 'signal' and try 
to avoid 
the term 'input rate' established in the literature on deterministic Hodgkin-Huxley models.

Depending on the value of the signal $a>0$, the following behaviour of the deterministic 
dynamical system is known, see  Ermentrout and Terman \cite[ pp.\ 63--66]{ET-10}. As there, 
see \eqref{internal_var_equil} and \eqref{F_infty} below, \eqref{det_HH} admits a unique 
equilibrium for every $a>0$. 
On some interval $(0,a_1)$ this equilibrium point is stable. There is a bistability interval 
$\mathbbm{I}_{\mathrm{bs}}=(a_1,a_2)$ on which a stable orbit coexists with a stable 
equilibrium 
point, and  
an interval $(a_2,a_3)$ on which the orbit is stable whereas the equilibrium point is unstable. 
As $a$ approaches from below the right endpoint $a_3$ of the last interval, orbits are 
collapsing towards  equilibrium; for  $a>a_3$ the equilibrium point is again stable. 
Here $0<a_1<a_2<a_3<\infty$ are suitably determined endpoints for intervals. Equilibrium 
points and orbits depend on the value of $a$. For biologically relevant values of the signal, 
evolution of the system along an orbit represents a remarkably fast 'large excursion' of all 
variables of the system, in particular of the membrane potential $V$, and is called a spike. 
Throughout the paper, we exclude unrealistically large values of the signal. 

In simulations --Euler schemes with time step $0.001$ where the starting point is selected at 
random, according to the uniform law on $(-12,120){\times}(0,1)^3$-- the equilibrium point 
appears to be globally attractive on  $(0,a_1)$. The orbit appears to be globally attractive  on  
$(a_2,a_3)$. On the bistability interval $\mathbbm{I}_{\mathrm{bs}}=(a_1,a_2)$, the 
behaviour of 
the 
system depends on the choice of the starting value: simulated trajectories either go to the 
equilibrium point,  or are attracted by the orbit.\footnote{
	Rinzel and Miller~\cite{RM-80} show that a branch of unstable periodic orbits exists on the 
	bistability interval, bifurcating below $a_2=\sup \mathbbm{I}_{\mathrm{bs}}$ and rejoining 
	the 
	stable orbits at $a_1=\inf \mathbbm{I}_{\mathrm{bs}}$ (in the sense of decreasing values of 
	$a$). 
	However, such orbits will not be seen in simulations with randomly chosen starting point.
} 

For our choice of the constants in equations \eqref{eq:F}--\eqref{eq:alphabeta} --those of 
Izhikevich \cite{Izh-07}, slightly different from both Ermentrout and Terman  \cite{ET-10} and 
Rinzel and Miller \cite{RM-80}-- simulations (here we refer to those\footnote{
	In unpublished work \cite{Hum-19}, Hummel simulated $1000$ trajectories with randomly 
	selected starting point for each value of the signal $a$ under consideration. Starting values 
	were sampled independently from the uniform law on $(-12,120){\times}(0,1)^3$. As a 
	function of $a$ (given in the first row of the tables below), the following relative number 
	(given in the second row) of trajectories was found to converge to an orbit. First, for $a$ in 
	$[8.0,9.0)$,  
	\begin{center}
		\begin{tabular}{l|l|l|l|l|l|l|l|l|l}
			\hline
			8.0 & 8.1 & 8.2 & 8.3 & 8.4 & 8.5 & 8.6 & 8.7 & 8.8 & 8.9 \\
			\hline
			0.986 & 0.991 & 0.993 & 0.998 & 1.000 & 1.000 & 1.000 & 1.000 & 1.000 & 1.000 \\
			\hline 
		\end{tabular}
	\end{center} 
	which determines an approximate location $\approx 8.4$  of the right endpoint of 
	$\mathbbm{I}_{\mathrm{bs}}$.     
	Then, considering values of $a$ in $[5.0,6.0)$  
	\begin{center}
		\begin{tabular}{l|l|l|l|l|l|l|l|l|l}
			\hline
			5.0 & 5.1 & 5.2 & 5.3 & 5.4 & 5.5 & 5.6 & 5.7 & 5.8 & 5.9 \\
			\hline
			0.000 & 0.000 & 0.000 & 0.751 & 0.777 & 0.803 & 0.821 & 0.835 & 0.847 & 0.863 \\
			\hline 
		\end{tabular}
	\end{center} 
	and looking in more detail into the interval $[5.2,5.3)$  
	\begin{center}
		\begin{tabular}{l|l|l|l|l|l|l|l|l|l}
			\hline
			5.20 & 5.21 & 5.22 & 5.23 & 5.24 & 5.25 & 5.26 & 5.27 & 5.28 & 5.29 \\
			\hline
			0.000 & 0.000 & 0.000 & 0.000 & 0.000 & 0.687 & 0.726 & 0.736 & 0.743 & 0.748 \\
			\hline 
		\end{tabular}
	\end{center} 
	the left endpoint of  $\mathbbm{I}_{\mathrm{bs}}$ is found between $5.24$ and $5.25$;   
	a closer look into $[5.24,5.25)$    
	\begin{center}
		\begin{tabular}{l|l|l|l|l|l|l|l|l|l}
			\hline
			5.240 & 5.241 & 5.242 & 5.243 & 5.244 & 5.245 & 5.246 & 5.247 & 5.248 & 5.249 \\
			\hline
			0.000 & 0.012 & 0.073 & 0.134 & 0.214 & 0.294 & 0.368 & 0.435 & 0.547 & 0.653 \\
			\hline 
		\end{tabular}
	\end{center} 
	shows that $\inf \mathbbm{I}_{\mathrm{bs}}$ is in fact very close to $5.24$. All values 
	above are 
	quoted from \cite{Hum-19}, p.\ 10 there. 
} 
done by \cite{Hum-19}) locate  $\inf \mathbbm{I}_{\mathrm{bs}} =a_1$ between $5.24$ and 
$5.25$, 
and $\sup \mathbbm{I}_{\mathrm{bs}}=a_2$ close to~$8.4$ . The value of $a_3$ is close 
to~$163.5$ 
and thus (given the shape of orbits when $a$ approaches $a_3$ from below) far beyond 
biological relevance.  Already at $a=5.5$,  about $80\%$ of all trajectories with randomly 
selected starting point  are attracted to the orbit, this  percentage being increasing in 
$a\in\mathbbm{I}_{\mathrm{bs}}$.

Equilibria for the deterministic Hodgkin-Huxley system \eqref{det_HH} can be determined as 
follows (\cite[pp.\ 38--39]{Izh-07}, and \cite{ET-10}). For a fixed value $v$ of the membrane 
potential, write 
\begin{equation}\label{internal_var_equil}
	\left( n_\infty(v) , m_\infty(v) ,  h_\infty(v) \right)  :=
	\left(\frac{\alpha_n}{\alpha_n+\beta_n}(v),   \frac{\alpha_m}{\alpha_m+\beta_m}(v),  
	\frac{\alpha_h}{\alpha_h+\beta_h}(v)  \right) 
\end{equation}
and define $F_\infty : \mathbb{R}\to\mathbb{R}$ by 
\begin{equation}\label{F_infty}
	F_\infty(v) := F\left(v, n_\infty(v), m_\infty(v),  h_\infty(v) \right). 
\end{equation}
Numerical evidence (see also the remarks in \cite[pp.\ 64--65]{ET-10})  shows that 
$F_\infty$ 
is strictly increasing on compacts. 
Thus, for signal $a>0$ in  \eqref{det_HH}--\eqref{eq:alphabeta}, solving 
\[
a \stackrel{!}{=} F_\infty(v^{\{a\}})
\]
we determine $v^{\{a\}}$ and thus the equilibrium point 
\begin{equation}\label{eq:equilibrium} 
	\left(v^{\{a\}} , n^{\{a\}} , m^{\{a\}} , h^{\{a\}}  \right) := 
	\left(v^{\{a\}}, n_\infty(v^{\{a\}}), m_\infty(v^{\{a\}}), h_\infty(v^{\{a\}})    \right) 
\end{equation}
of the deterministic system \eqref{det_HH} with signal $a>0$.

\begin{figure}[h]
	\centering
	\includegraphics[width=0.95\linewidth]{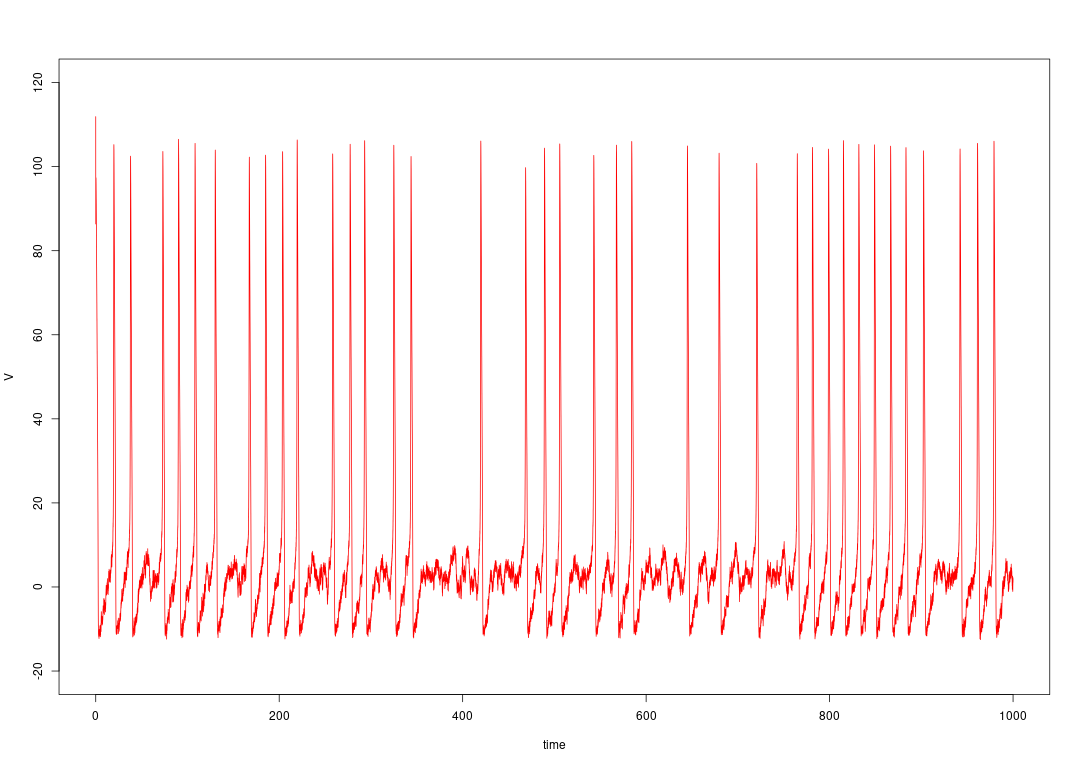}
	\caption{Membrane potential in a simulated stochastic Hodgkin-Huxley neuron 
	$\mathbb{X}^{(\vartheta,\tau,\sigma)}$. The value of the signal is $\vartheta=4$, and the 
	parameters for the Ornstein Uhlenbeck process $X$ are $\tau=0.5$ and $\sigma=2.5$. 
	Initial conditions are selected at random, according to the stationary law of $X$ and  
	according to the uniform law on $(-12,120){\times}(0,1)^3$ for $(V,n,m,h)$. 
		The simulation was done using an Euler scheme with equidistant steps \(0.001\).  
		Under signal $a=\vartheta=4$, a deterministic Hodgkin-Huxley neuron \eqref{det_HH} 
		would be attracted to the stable equilibrium point  \eqref{eq:equilibrium}.  In the 
		stochastic Hodgkin-Huxley neuron \eqref{def_Y_HH}--\eqref{stoch_HH} of 
		Section~\ref{sect:stoch_HH}, 'noise' -- in the combination of parameters considered 
		here--  turns 
		out to be strong enough to create frequent spikes.
		Here and in all graphics below, no attempt is made towards  'biologically relevant scaling' 
		of the time axis.
	} 
\label{fig:1a}
\end{figure}

\section{Stochastic Hodgkin-Huxley with constant signal}\label{sect:stoch_HH} 

Prepare an Ornstein-Uhlenbeck process with back-driving force $\tau>0$ and volatility 
$\sigma>0$ 
\begin{equation}\label{def_X_HH}
	dX_t = - \tau X_t dt + \sigma dW_t. 
\end{equation}
In order to feed noise into the system \eqref{det_HH}, we replace  $a dt$ in the first 
equation of the deterministic model \eqref{det_HH} by increments $dY_t$ of a stochastic 
process  
\begin{equation}\label{def_Y_HH}
	Y_t = \vartheta t  + X_t  \quad,\quad t\ge 0 
\end{equation}
with some constant $\vartheta>0$, unique strong solution to\footnote{
	SDE \eqref{sde_Y_HH} for the accumulated input $(Y_t)_t$ has remarkable statistical 
	consequences: 
	in the stochastic Hodgkin-Huxley model with constant signal, knowing $\mathbb{X}_0$ and 
	observing the membrane potential $V$ continuously in time, the signal  $\vartheta$ can be 
	estimated at a better rate than the back-driving force $\tau$
	(\cite{Hoe-21}, corollary 2 in section 4). 
} 
the stochastic differential equation 
\begin{equation}\label{sde_Y_HH}
	dY_t = \vartheta( 1 + \tau t ) dt   -   \tau  Y_t dt  + \sigma dW_t. 
\end{equation}
Together with equation \eqref{def_Y_HH} or \eqref{sde_Y_HH}, the system 
\begin{equation}\label{stoch_HH}
	\left\{
	\begin{array}{lll}
		dV_t &= & dY_t - F(V_t,n_t,m_t,h_t) dt \\
		dn_t &= & [ \alpha_n(V_t) (1-n_t) - \beta_n(V_t) n_t ] dt \\
		dm_t &= & [ \alpha_m(V_t) (1-m_t) - \beta_m(V_t) m_t ] dt \\
		dh_t &= & [ \alpha_h(V_t) (1-h_t) - \beta_h(V_t) h_t ] dt
	\end{array} 
	\right.     
\end{equation}
defines a stochastic Hodgkin-Huxley model. We speak of $\vartheta>0$ as the 'signal' 
encoded 
in  the system. 
In contrast to the deterministic case, the behaviour of the biological variables 
\eqref{stoch_HH} in the stochastic system is not only governed by the value of the signal 
$\vartheta$, but also depends on the level of 'noise', \textit{i.e.} on the values of the volatility 
$\sigma$ 
and the back-driving force $\tau$ in the Ornstein-Uhlenbeck process \eqref{def_X_HH}. 
We thus consider the  $5$-dimensional strong Markov process 
\begin{equation}\label{doublescript_X}
	\mathbb{X}^{(\vartheta,\tau,\sigma)} = (\mathbb{X}^{(\vartheta,\tau,\sigma)}_t)_{t\ge 0}  
	\quad,\quad  \mathbb{X}^{(\vartheta,\tau,\sigma)}_t := 
	\left( V_t , n_t , m_t , h_t   , X_t \right)_{t\ge 0} 
\end{equation}
having state space $E:=\mathbb{R}\times[0,1]^3\times\mathbb{R}$. $E$ is endowed with its 
Borel-$\sigma$-field $\mathcal{E}$. The process \eqref{doublescript_X} is homogeneous in 
time, with encoded signal $\vartheta$ and semigroup  
\[
(P^{(\vartheta,\tau,\sigma)}_t)_{t\ge 0}
\]
of transition probabilities on $(E,\mathcal{E})$. We suppress superscripts when the context is 
clear.

A biological interpretation of the system \eqref{doublescript_X} is as follows. Assume that the 
neuron which we consider is part of a large and active network. Then a structure 
$dY_t=\vartheta dt + dX_t$ of input  reflects  superposition of some global level 
$\vartheta>0$ of excitation in the network with noise in the single neuron. Noise in the single 
neuron arises as a result of accumulation and decay of a large number of small postsynaptic 
charges, caused by incoming spikes --registered at synapses, excitatory or inhibitory, present 
in large number and in complex spatial distribution along the dendritic tree of the neuron, then 
undergoing decay and finally being summed up--  which the neuron receives from a large 
number of other neurons active within the same network.

Throughout the paper, we exclude by convention unrealistically large values of the signal 
$\vartheta$: 
orbits in a deterministic system with same value of the signal always should admit a biological  
interpretation in terms of a spike. Even if we write '$\vartheta>0$' below, this is the same 
caveat as in Section~\ref{sect:det_HH}.

\subsection{Positive Harris recurrence}\label{subsect:harris}

We discuss ergodicity properties of  systems  \eqref{doublescript_X}. 
For stochastic Hodgkin-Huxley models encoding signals which are deterministic periodic 
functions, positive Harris recurrence is established in  H\"opfner, L\"ocherbach and Thieullen 
\cite{HLT-16,HLT-17}, see also \cite{HLT-16a}, and including more general settings in 
Holbach \cite{Hol-20}.  Our case 
of constant signal $\vartheta>0$ is then essentially a corollary. For background on Harris 
recurrence see Nummelin \cite{Num-78, Num-85}, Azema, Duflo and Revuz 
\cite{ADR-69}, Revuz and Yor \cite{RY-91}, H\"opfner and L\"ocherbach \cite{HL-03}.

\begin{theorem}\label{theo:21}
The following holds for every $\vartheta>0$, $\tau>0$, $\sigma>0$ : 

\begin{enumerate}[a)]
\item \label{item:21a}
The process $(\mathbb{X}^{(\vartheta,\tau,\sigma)}_t)_{t\ge 0}$ is positive Harris 
recurrent. 
\item \label{item:21b} For arbitrary step size $0<T<\infty$, grid chains 
$(\mathbb{X}^{(\vartheta,\tau,\sigma)}_{kT})_{k\in\mathbbm{N}_0}$ are positive Harris 
recurrent. 
\item \label{item:21c} For arbitrary step size $0<T<\infty$, chains of path segments 
\[
(\mathbb{X}^{(\vartheta,\tau,\sigma)}_{[kT,(k+1)T]})_{k\in\mathbbm{N}_0} \quad,\quad 
\mathbb{X}^{(\vartheta,\tau,\sigma)}_{[kT,(k+1)T]} := 
(\mathbb{X}^{(\vartheta,\tau,\sigma)}_t)_{ kT \le t \le (k+1)T }
\]
with values in the space of continuous functions  $C([0,T],E)$
are positive Harris recurrent. 
\item  \label{item:21d} For every $0<T<\infty$, there is some 'small set' $C\in\mathcal{E}$ of 
strictly positive 
invariant measure, some probability law $\nu$ on $(E,\mathcal{E})$, and some $\alpha\in 
(0,1)$ such that Nummelin's minorization  condition holds:   
\[
P^{(\vartheta,\tau,\sigma)}_T(x,dy) \ge \alpha \mathbbm{1}_C(x) \nu(dy) \quad\mbox{ for 
all 
}  x, y 
\mbox{ in } E.  
\]
\end{enumerate}
\end{theorem}

\begin{proof}
	Fix $(\vartheta,\tau,\sigma)$ and write the system \eqref{doublescript_X} as
	\begin{equation}\label{rewriting_stoch_HH}
		\left\{
		\begin{array}{lll}
			dV_t &= & dX_t - [F-\vartheta](V_t,n_t,m_t,h_t) dt \\
			dj_t &= & [ \alpha_j(V_t) (1-j_t) - \beta_j(V_t) j_t ] dt 
			\quad,\quad j\in\{n,m,h\}. 
		\end{array} 
		\right.     
	\end{equation}
	Then \eqref{rewriting_stoch_HH} amounts to a simplified variant of the OU-type 
	Hodgkin-Huxley systems investigated in \cite{HLT-16,HLT-17}:  we can replace 
	the function  $F$ there by $\widetilde{F}:=F{-}\vartheta$, a change which does not affect 
	the 
	proofs in \cite{HLT-16} and \cite{HLT-17}, and then encode $\widetilde{S}\equiv 0$ in place 
	of the 
	deterministic periodic function into the drift of the diffusion process in \cite{HLT-16} and  
	\cite{HLT-17}.  
	This allows to view any $0<T<\infty$ as a period for our stochastic system 
	\eqref{doublescript_X};  the coefficients remain real analytic. 
	
	Now  \ref{item:21a}) and \ref{item:21b}) correspond to Theorems 2.7 and 2.2 in 
	\cite{HLT-16}. The lower bound  \ref{item:21d}) 
	corresponds to 
	Theorem~4 and Corollary~1 (together with Sections \ref{subsection:53} 
	--\ref{subsection:55}) in 
	\cite{HLT-17}, 
	or to step 1) in the proof to Theorem 2.9 in \cite{HLT-16}. Assertion \ref{item:21c}) on path 
	segments 
	follows from \ref{item:21b}) as in Theorem~2.1 of H\"opfner and Kutoyants \cite{HK-10}. 
\end{proof}

Let $Q^{(\vartheta,\tau,\sigma)}_x$ denote the law of the process 
$(\mathbb{X}^{(\vartheta,\tau,\sigma)}_t)_{t\ge 0}$ starting from $x\in E$, a probability 
measure on the canonical path space $(C,\mathcal{C})$ of continuous functions $[0,\infty)\to 
E$. We equip $(C,\mathcal{C})$  with the right-continuous filtration 
$\mathbb{G}=(\mathcal{G}_t)_{t\ge 0}$ generated by the canonical process. This allows to 
view the single neuron $\mathbb{X}^{(\vartheta,\tau,\sigma)}$  in \eqref{doublescript_X} as 
a canonical process on a  canonical path space under $Q^{(\vartheta,\tau,\sigma)}_x$.  As 
usual, 'almost surely' means  $Q^{(\vartheta,\tau,\sigma)}_x$-almost surely for every $x\in E$.

Positive Harris recurrence Theorem~\ref{theo:21}~\ref{item:21a})+\ref{item:21b}) implies that 
there exists a unique 
invariant probability  
$\mu^{(\vartheta,\tau,\sigma)}$ on the state space $(E,\mathcal{E})$, that sets of positive 
invariant probability are visited infinitely often (for events $F\in\mathcal{E}$ and arbitrary 
$0<T<\infty$, 
\[
\mbox{if}\quad\mu^{(\vartheta,\tau,\sigma)}(F) > 0 \quad : \quad \int_0^\infty 
\mathbbm{1}_F(\mathbb{X}^{(\vartheta,\tau,\sigma)}_s) ds = \infty 
\quad,\quad 
\sum_{k=1}^\infty \mathbbm{1}_F(\mathbb{X}^{(\vartheta,\tau,\sigma)}_{kT}) = \infty
\]
almost surely), and implies strong laws of large numbers: for functions $h:E\to\mathbb{R}$ 
which belong to $L^1(\mu^{(\vartheta,\tau,\sigma)})$, limits 
\begin{equation}\label{as_limits_1}
	\frac{1}{t} \int_0^t h(\mathbb{X}^{(\vartheta,\tau,\sigma)}_s)  ds \longrightarrow  \int_E h 
	d\mu^{(\vartheta,\tau,\sigma)}
	\quad,\quad  t\to\infty 
\end{equation}
and, for every  $0<T<\infty$ fixed, 
\begin{equation}\label{as_limits_2}
	\frac{1}{n} \sum_{k=1}^n h(\mathbb{X}^{(\vartheta,\tau,\sigma)}_{kT})   \longrightarrow  
	\int_E 
	h  d\mu^{(\vartheta,\tau,\sigma)}
	\quad,\quad  n\to\infty  
\end{equation}
exist almost surely. Consider also $Q^{(\vartheta,\tau,\sigma)}_x$ restricted to 
$(C_T,\mathcal{C}_T)$ where $C_T$ is the path space of continuous functions 
$[0,T]\to\mathbb{R}$, and write 
\[
Q^{(\vartheta,\tau,\sigma)}_\mu   := Q^{(\vartheta,\tau,\sigma)}_{ 
\mu^{(\vartheta,\tau,\sigma)} } = \int_E \mu^{(\vartheta,\tau,\sigma)}(dx) 
Q^{(\vartheta,\tau,\sigma)}_x   
\]
for the probability law on $(C,\mathcal{C})$ or on $(C_T,\mathcal{C}_T)$ under which the 
canonical process $\mathbb{X}$ on $(C,\mathcal{C})$ or on $(C_T,\mathcal{C}_T)$ is a 
stationary process. 
If for some $T$ a function $g:C_T\to\mathbb{R}$ belongs to  
$L^1(Q^{(\vartheta,\tau,\sigma)}_\mu)$, then 
\begin{equation}\label{as_limits_3}
	\frac{1}{n} \sum_{k=0}^{n-1} 
	g\left(\mathbb{X}^{(\vartheta,\tau,\sigma)}_{[kT,(k+1)T]}\right)   
	\longrightarrow 
	\int_{C_T} g d Q^{(\vartheta,\tau,\sigma)}_\mu   
	\quad,\quad  n\to\infty  
\end{equation}
holds almost surely,   by positive Harris recurrence Theorem~\ref{theo:21}~\ref{item:21c}) for 
path 
segment chains.

\subsection{Spike times and spiking patterns }\label{sect:output_HH}

In a  stochastic Hodgkin-Huxley neuron \eqref{doublescript_X}, 
we define 'beginning' of a spike as the time of upcrossing of the $m$-variable over the 
$h$-variable, and  'end' of the same spike as the time of re-downcrossing of $m$ under $h$, 
as in (2.15) in \cite{HLT-16}: the membrane potential $V$ reaches its maximum on this time 
interval almost immediately after the  upcrossing of $m$ over $h$. We 
define the spike time as the time of the beginning
of a spike (the time at which the membrane potential attains a local 
maximum does not have the structure of a stopping time). Then the spike train emitted by the 
stochastic Hodgkin-Huxley neuron \eqref{doublescript_X} 
is the sequence $(\tau_j)_{j\ge 1}$ of  $\mathbb{G}$-stopping times    
\begin{equation}\label{def_spiketimes}
	\tau_j := \inf\{ t>\sigma_{j-1} : m_t>h_t \}, \sigma_j := \inf\{ t>\tau_j+\delta_0 : 
	m_t<h_t \}, j\ge 1,\sigma_0=\tau_0=0   
\end{equation} 
with convention $\inf\{\emptyset\} = \infty$, and with $\delta_0>0$ arbitrarily small but fixed. 
The sequence $(\tau_j)_j$ is strictly increasing and tends to $\infty$;  we associate the 
counting process   
\begin{equation}\label{def_N}
	N = (N_t)_{t\ge 0} \quad,\quad 
	N_t := \sum_{j\ge 1} \mathbbm{1}_{(0,t]}(\tau_j). 
\end{equation}   
Interspike times $( \tau_j-\tau_{j-1} )_{j\ge 1}$ have no reason to be independent or 
identically distributed, and $N$ has no reason to be a Poisson process (in particular, for every 
$t$, the random variable $N_t$ is bounded by construction). This does not exclude the 
possibility that on 
compact time intervals, under certain parameter configurations, $N$ with large probability 
may 
look quite similar to a Poisson process. 

~
\begin{proposition} \label{prop:22}
~
\begin{enumerate}[a)]
\item 
\label{item:22a}
For $0<T<\infty$ fixed, almost surely within the family of time 
intervals $\{ [kT,(k{+}1)T] : k\in\mathbbm{N}_0 \}$,   
an infinite number of intervals will contain spikes and an infinite number of intervals will remain 
spikeless.
\item \label{item:22b}
The empirical distribution functions $\widehat{H}_n$ associated to the first $n$ observed 
interspike times 
\[
\tau_{\ell+1}{-}\tau_\ell  \quad,\quad 1\le \ell\le n 
\]
converge almost surely as $n\to\infty$, uniformly on $[0,\infty)$, to the distribution function 
$H^{(\vartheta,\tau,\sigma)}$ of some probability law which is concentrated on $(0,\infty)$. 
\end{enumerate}
\end{proposition}

\begin{proof} 
	As in the proof of Theorem~\ref{theo:21}, \ref{item:21a}) and \ref{item:21b}) correspond to 
	Theorems 2.8 and 2.9 in 
	\cite{HLT-16}. 
\end{proof}
In the stationary regime, the Laplace transform of the number of spikes observed on path 
segments 
of length $T$
\begin{equation}\label{LT_NT_stationary}
	\psi^{(\vartheta,\tau,\sigma)}_T(\lambda) := 
	E_\mu^{(\vartheta,\tau,\sigma)}\left( e^{ - 
	\lambda N_T} \right) \quad,\quad \lambda\ge 0 
\end{equation}
and the probability that a path segment of length $T$ contains less than $v$ spikes  
\begin{equation}\label{DF_NT_stationary}
	F^{(\vartheta,\tau,\sigma)}_T(v) := Q_\mu^{(\vartheta,\tau,\sigma)}\left(  N_T \le 
	v \right) 
	\quad,\quad v\ge 0  
\end{equation}
are of interest for statistical purposes. Whereas there is no hope to get explicit expressions for 
the left hand sides of \eqref{LT_NT_stationary} or \eqref{DF_NT_stationary}, Harris 
recurrence provides us with empirical Laplace transforms and  empirical distribution functions. 
\newpage
\begin{proposition}\label{prop:23}
Under $(\vartheta,\tau,\sigma)$, for $0<T<\infty$ fixed,   
\begin{enumerate}[a)]
	\item \label{item:33a} the functions  
	\[
	\widehat{\psi}_{n,T}(\lambda) := \frac{1}{n} \sum_{k=1}^n e^{-\lambda ( N_{kT} -N_{(k-1)T} 
		) 
	} \quad,\quad \lambda\ge 0 
	\quad,\quad  n\in\mathbbm{N}    
	\]
	converge  uniformly on $[0,\infty)$, almost surely as $n\to\infty$, to the Laplace transform 
	$\psi^{(\vartheta,\tau,\sigma)}_T$ in \eqref{LT_NT_stationary};  
	\item \label{item:33b} the functions 
	\[
	\widehat{F}_{n,T}(v) := \frac{1}{n} \sum_{k=1}^n  \mathbbm{1}_{\{N_{kT} -N_{(k-1)T}  \le 
		v\}} 
	\quad,\quad v\ge 0
	\quad,\quad  n\in\mathbbm{N}   
	\]
	converge uniformly on $[0,\infty)$, almost surely as $n\to\infty$, to the distribution function 
	$F^{(\vartheta,\tau,\sigma)}_T$ in \eqref{DF_NT_stationary}.
\end{enumerate}
\end{proposition}
\begin{proof}
	For $T$, $v$ and $\lambda$ fixed, both $e^{ -\lambda (N_{kT} -N_{(k-1)T}) }$ or 
	$\mathbbm{1}_{\{ N_{kT} -N_{(k-1)T} \le v \}} $ are bounded functionals 
	$h\left( \mathbb{X}_{[(k-1)T,kT ]} \right)$ of paths segments, and pointwise 
	convergence almost surely  holds in virtue of Theorem~\ref{theo:21}~\ref{item:21c}) and 
	\eqref{as_limits_3}: 
	\[
	\frac{1}{n} \sum_{k=1}^n  h\left( \mathbb{X}_{[(k-1)T,kT]}  \right)
	\longrightarrow E_\mu^{(\vartheta,\tau,\sigma)}\left( h(\mathbb{X}_{[0,T]}) \right) 
	\quad,\quad n\to\infty. 
	\]
	The limit functions $F^{(\vartheta,\tau,\sigma)}_T$ of \eqref{DF_NT_stationary} and 
	$\psi^{(\vartheta,\tau,\sigma)}_T$ of \eqref{LT_NT_stationary} are monotonous and 
	bounded, so uniformity on $[0,\infty)$ follows as in the classical proof of the 
	Glivenko-Cantelli Theorem. 
\end{proof}

\begin{proposition}\label{prop:24}
Under $(\vartheta,\tau,\sigma)$, as $t\to\infty$, the limit 
\[
\lim_{t\to\infty} \frac{1}{t} N_t = E_\mu^{(\vartheta,\tau,\sigma)}\left( N_1 \right)
\]
exists almost surely. 
\end{proposition}
\begin{proof} 
	Fix  $T:=1$. View $N_{kT} -N_{(k-1)T}$ as a functional 
	$h\left( \mathbb{X}_{[(k-1)T,kT]} \right)$ of the paths segments; by construction in 
	\eqref{def_spiketimes},   $N_{kT} -N_{(k-1)T}$ being bounded by $\frac{1}{T} \delta_0$, 
	this  
	functional $h : C_T \to [0,\infty)$ is bounded. Theorem~\ref{theo:21}~\ref{item:21c}) and 
	\eqref{as_limits_3} give 
	almost sure convergence 
	\[
	\frac{1}{n} N_n = \frac{1}{n} \sum_{k=1}^n  h\left( \mathbb{X}_{[k-1,k]}  \right)
	\longrightarrow 
	E_\mu^{(\vartheta,\tau,\sigma)}\left( N_1 \right) 
	\]
	under $(\vartheta,\tau,\sigma)$, and with $\lfloor t \rfloor \le t \le \lfloor  t \rfloor + 1 $ the 
	assertion follows. 
\end{proof}

The following extension of Proposition~\ref{prop:22}~\ref{item:22b}) allows to consider 
spiking 
patterns. 

\begin{theorem}\label{theo:25}
For every $L\in\mathbbm{N}$, empirical distribution functions 
$\widehat{G}_m : [0,\infty)^L \to [0,1]$ associated to the first $m$ observed $L$-tuples of 
successive interspike times 
\begin{equation}\label{Ltupels_ISI}
	\left(\tau_{n+1} - \tau_{n},  \ldots,  \tau_{n+L} - \tau_{n+L-1}\right) ,\quad 
	n\in \mathbbm{N}
\end{equation}
converge almost surely as $m\to\infty$, uniformly on $[0,\infty)^L$,  to the distribution 
function $G^{(\vartheta,\tau,\sigma)}_\mu$ of some probability law concentrated on 
$(0,\infty)^L$. 
\end{theorem}

The proof, based on renewal techniques which extend the proof of 
Proposition~\ref{prop:22}~\ref{item:22b}) above 
(\textit{i.e.} the proof of Theorem~2.9 in  \cite{HLT-16}), is given together with some 
complements in 
the appendix Section~\ref{sect:proof_prop_2.5}. The probability law  
$G^{(\vartheta,\tau,\sigma)}_\mu $ in Theorem~\ref{theo:25} governs the variety of typical 
patterns on 
$(0,\infty)^L$ which will appear in the long run in  $L$-tuples of successive interspike times.

\subsection{Output of a stochastic Hodgkin-Huxley neuron}\label{output}

The counting process $N$ in \eqref{def_N} allows to measure the accumulated activity of the 
neuron $\mathbb{X}^{(\vartheta,\tau,\sigma)}$ in \eqref{doublescript_X} by a stochastic 
process $U$ which we call 'output'
\begin{equation}\label{def_output_process}
	U=(U_t)_{t\ge 0} \quad,\quad   dU_t = -c_1 U_{t-} dt + dN_t \quad,\quad U_0= 0   
\end{equation}
where $c_1$ is some constant, strictly positive and finite. We have 
\[
U_t - U_s =  U_s e^{- c_1 (t-s) }  +  \int_{(s,t]}  e^{- c_1 (t-v) } dN_v    
 =   U_s e^{- c_1 (t-s) }  +  \sum_{j\ge 1} \mathbbm{1}_{(s,t]}(\tau_j) e^{ - c_1 (t-\tau_j) 
}    
\]
for $0\le s<t$, and in particular at the spike times 
\begin{equation}\label{output_properties}
	U_{\tau_\ell} = \sum_{j=1}^\ell e^{ - c_1 (\tau_\ell-\tau_j) } \quad,\quad 
	U_{\tau_j -} = U_{\tau_{j-1}} e^{ - c_1 (\tau_j-\tau_{j-1}) } \quad,\quad 
	U_{\tau_k} =  U_{\tau_k -} +  1    . 
\end{equation}

Properties of the output process in the long run can be discussed as an application  of 
Theorem~\ref{theo:25}.

\begin{proposition}\label{prop:26}
For $\varepsilon>0$ choose $L$ large enough so that  $\sum_{\ell>L} e^{ - 
c_1 \delta_0  \ell } < \varepsilon$. Then pairs 
\begin{equation}\label{Uafterspike_Ubeforespike_neu}
	\left(  U_{\tau_{n+L}} ,  U_{(\tau_{n+L+1})^-}\right) \quad,\quad n\in\mathbbm{N}
\end{equation}
admit approximations  
\begin{equation}\label{approx_afterspike_approx_beforespike} 
	\left(V_n,  V_{n+1}^- \right) \quad:\quad 
	V_n := \sum_{j=n}^{n+L} e^{-c_1(\tau_{n+L}-\tau_j) } \quad,\quad 
	V_{n+1}^- := \sum_{j=n}^{n+L} e^{- c_1 (\tau_{n+L+1}-\tau_j) }  
\end{equation}
with the following properties: we have bounds uniformly in $n$
\[
\sup \left\{ | U_{\tau_{n+L}} - V_n | , | U_{(\tau_{n+L+1})^-} - V_{n+1}^- | :   
n\in\mathbbm{N} \right\} < \varepsilon, 
\]
and empirical distribution functions associated to the first $m$ pairs out of 
\eqref{approx_afterspike_approx_beforespike} 
converge almost surely as $m\to\infty$, uniformly on  $[0,\infty)^2$, to the distribution 
function of some probability law which is concentrated on $(0,\infty)^2$.  
\end{proposition}

The  proof of Proposition~\ref{prop:26} is also shifted to the appendix 
Section~\ref{sect:proof_prop_2.5}. 
Note that in order to obtain small values of $\varepsilon$ in Proposition~\ref{prop:26} we have 
to  require huge 
values of $L$, so the result seems more of theoretical than of practical interest. With the same 
technique of proof, the result can be extended to $(J+1)$-tuples 
\begin{equation}\label{Uafterspike_Ubeforespike_extended}
	\left(
	\left(U_{\tau_{n+j+L}},  U_{(\tau_{n+j+L+1})^-}\right)_{ 0 \le j \le J } 
	\right) ,\quad    n\in\mathbbm{N} 
\end{equation}
for $J\in\mathbbm{N}$ arbitrary but fixed.

\begin{figure}[h]
	\centering
	\includegraphics[width=0.95\linewidth]{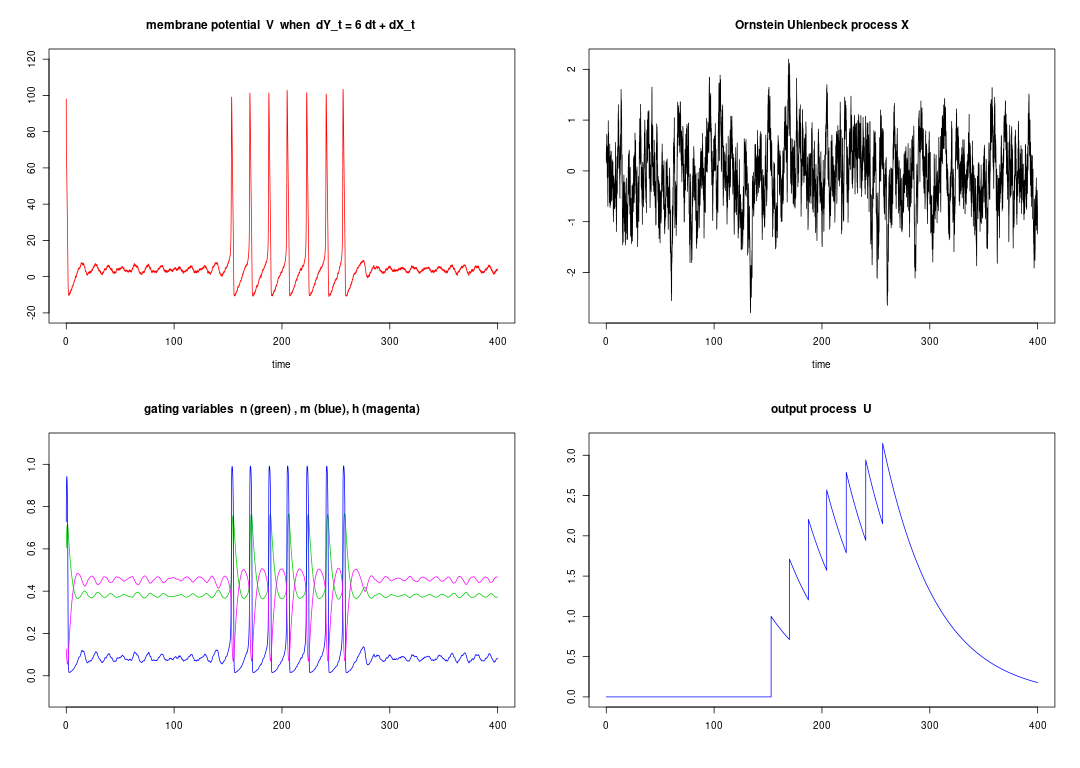} 
	\caption{Simulated trajectory of a stochastic Hodgkin-Huxley neuron 
	$\mathbb{X}^{(\vartheta,\tau,\sigma)}$. The signal is $\vartheta=6$. The parameters for the 
	Ornstein Uhlenbeck process $X$ are $\tau=0.7$ and $\sigma=0.83666$. The parameter in 
	the output process $U$ is $c_1 = 0.02$. Initial conditions are selected at random, according 
	to the stationary law of $X$ and  according to the uniform law on $(-12,120){\times}(0,1)^3$ 
	for $(V,n,m,h)$. The starting value for the output process $U$ is $0$. 
		The simulation was done using an Euler scheme with equidistant steps \(0.001\). 
	} 
	\label{fig:1}
\end{figure}

\section{Quiet behaviour of stochastic neurons}\label{sect:quiet}

In a deterministic Hodgkin-Huxley neuron, sufficiently small values of the signal $\vartheta$ 
--smaller than $\inf( \mathbbm{I}_{\mathrm{bs}} )$, see Section~\ref{sect:det_HH}-- grant 
that 
trajectories  are attracted to the stable equilibrium point (randomly chosen initial conditions). 
In a stochastic Hodgkin-Huxley neuron, by Proposition~\ref{prop:22}, spikes will occur almost 
surely also for small values of the signal $\vartheta$. Simulations under $\vartheta < \inf( 
\mathbbm{I}_{\mathrm{bs}} )$ show that the spiking behaviour --in form of single isolated 
spikes or 
small groups of spikes-- depends on some interplay between the volatility $\sigma$ and 
the back-driving force $\tau$.

In Definition~\ref{def:31} below, we shall define  quiet behaviour of stochastic 
Hodgkin-Huxley 
neurons as an event on which  spike trains observed over a long time interval seem close to a 
Poisson process with low intensity. 
The process $N$ counting spikes has been defined in~\eqref{def_spiketimes}.

To random variables $\xi_1,\ldots,\xi_n$,  $\mathbbm{N}_0$-valued but which in general we 
do not assume independent or identically distributed, we associate an empirical distribution 
function  $\widehat{F}_n(v)=\frac{1}{n} \sum_{j=1}^n \mathbbm{1}_{[0,v]}(\xi_j)$,  an 
empirical 
Laplace 
transform $\widehat{\psi}_n(v):=\frac{1}{n} \sum_{j=1}^n e^{-v \xi_j}$, $v\ge 0$, and an 
empirical mean  $\bar\xi_n := \frac{1}{n} \sum_{j=1}^n \xi_j$.

For Poisson random variables $\xi$ with parameter $\lambda > 0$, we write 
$\mathsf{F}_\lambda(v)=\mathsf{P}_\lambda(\xi\le v)$ for the distribution function~(DF), 
$v\ge 0$, 
${\phi}_\lambda(v)=\mathsf{E}_\lambda(e^{-v \xi})$ for the Laplace transform (LT), and 
\begin{equation}\label{poisson_quantiles}
	\bar{\mathsf{q}}(\alpha,\lambda):={\rm min}\{n\in\mathbbm{N}_0 : 
	\mathsf{P}_\lambda(\xi>n)\le\alpha \}
\end{equation}
for upper $\alpha$-quantiles. Write $\mathsf{P}_\lambda^n$ for the joint law of 
i.i.d.\ Poisson random variables $(\xi_1,\ldots,\xi_n)$  with parameter $\lambda>0$, 
$\widehat{\mathsf F}_n$ for the empirical distribution function, $\widehat{\phi_n}$ for the 
empirical distribution function, and $\bar\xi_n$ for the empirical mean. In order to obtain 
quantified benchmarks for comparison with other data sets we shall use laws  
\begin{equation}\label{poisson_stat_DF_LT}
	\mathcal{L}\left( \int_I \left| \widehat{\mathsf{F}_n}(v) - \mathsf{F}_{\bar\xi_n}(v) \right| dv 
	\mid
	{\mathsf{P}}_\lambda^n \right) \quad,\quad
	\mathcal{L}\left( \int_I \left| \widehat{\phi_n}(v) - \phi_{\bar\xi_n}(v) \right| dv \mid 
	\mathsf{P}_\lambda^n \right)    
\end{equation}
where $I$ is some closed interval in $[0,\infty)$ whose left endpoint is  $0$, and their upper 
$\alpha$-quantiles 
\begin{equation}\label{poisson_stat_qu_DF_LT}
	\bar{\mathsf{q}}_{DF} ( \alpha; \lambda, n, I ) \quad,\quad \bar{\mathsf{q}}_{LT} ( \alpha; 
	\lambda, n, I ) . 
\end{equation}
We shall determine laws \eqref{poisson_stat_DF_LT} and quantiles 
\eqref{poisson_stat_qu_DF_LT} empirically using simulations.

On the basis of ergodicity established in Section~\ref{sect:stoch_HH}  and motivated in 
particular 
by 
\eqref{as_limits_3} and Theorem~\ref{theo:21}~\ref{item:21c}) applied to path segments of 
sufficient 
length $T_0$, the following definition counts spikes on successive segments   
\[
\mathbb{X} \mathbbm{1}_{ \rrbracket (k-1)T_0 , kT_0 \rrbracket }   \quad,\quad 1\le k\le K 
\]
where $K$ is assumed to be large. We call a stochastic Hodgkin-Huxley neuron quiet when 
\ref{item:1_221221}) 
and \ref{item:2_221221}) hold:  
\begin{enumerate}[i)]
	\item \label{item:1_221221} a Poisson-goodness of fit test does not reject a Poisson 
	hypothesis with estimated 
	parameter;
	\item \label{item:2_221221}  the estimated parameter is small enough. 
\end{enumerate}

\begin{definition}\label{def:31}
Assume that a stochastic Hodgkin-Huxley neuron \eqref{doublescript_X} 
with parameters $(\vartheta,\tau,\sigma)$ has been observed over a long time interval 
$[0,KT_0]$, $K\in\mathbbm{N}$. For $K$ and $T_0$ large enough, put  
\begin{equation}\label{K_Tnull_setting} 
	T_1 := K  T_0 \quad,\quad 
	\widetilde{\lambda} := \frac{1}{T_1} N_{T_1}  \quad , \qquad 
	\xi_k :=  N_{kT_0} - N_{(k-1)T_0}    \quad,\quad 1\le k\le K      
\end{equation}
and fix critical values  
\begin{equation}\label{crit_val_GoF}
	\lambda_c := 0.0005 \quad,\quad \alpha_c:=0.0005  
\end{equation}
for hypothetical Poisson intensities and  quantiles.   
Let $Q(T_0,K)$ denote the event in $\mathcal{G}_{T_1}$ on which either: spikes are 
extremely rare, \textit{i.e.}   
\begin{equation}\label{extremely_few}
	N_{T_1} \le 2 \quad\mbox{and}\quad \widetilde{\lambda} \le  0.0001    , 
\end{equation}
or: with quantiles \eqref{poisson_quantiles}, at most 
\begin{equation}\label{cond_**}
	N_{T_1} \le  \bar{\mathsf{q}} \left(0.05 , \lambda_c T_1 \right)   
\end{equation}
spikes occur, and their location on the time axis is such that  increments $\xi_1, \ldots,\xi_K$ 
in \eqref{K_Tnull_setting} are in good fit with i.i.d. Poisson random variables with estimated 
parameter $\widetilde{\lambda} T_0 $,  in the following sense: 
with $I:=[0,5.5]$ and with $\widehat{F}_K$ and $\widehat{\psi}_K$ defined from 
$\xi_1,\ldots,\xi_K$ we 
use the statistics  
\[
\Delta_{DF}(T_0,K)  :=   \int_I \left| \widehat{F_K}(v) - \mathsf{F}_{\widetilde{\lambda} T_0}(v) 
\right| dv 
\quad,\quad 
\Delta_{LT}(T_0,K)  :=   \int_I \left| \widehat{\psi_K}(v) - \phi_{\widetilde{\lambda} T_0}(v) 
\right| dv   
\]
($\mathcal{G}_{T_1}$-measurable) and require, with critical values \eqref{crit_val_GoF} and 
quantiles \eqref{poisson_stat_qu_DF_LT},  that the following holds: 
\begin{equation}\label{cond_+}
	\Delta_{DF}(T_0,K)  \le c_{DF} 
	\quad\mbox{with}\quad    c_{DF} :=  \bar{\mathsf{q}}_{DF} \left( \alpha_c  ; \lambda_c 
	T_0 , 
	K 
	, I \right) ,
\end{equation}
\begin{equation}\label{cond_++}
	\Delta_{LT}(T_0,K)  \le c_{LT} 
	\quad\mbox{with}\quad   c_{LT} :=  \bar{\mathsf{q}}_{LT} \left( \alpha_c ; \lambda_c T_0 
	, K 
	, 
	I \right) .
\end{equation}
On events $Q(T_0,K)\in\mathcal{G}_{T_1}$ we call the stochastic neuron $\mathbb{X}$ in 
\eqref{doublescript_X} \textbf{quiet}. 
\end{definition}

In the setting of Section~\ref{sect:stoch_HH}, spike trains are never exactly Poisson 
(interspike 
times being  $>\delta_0$ by con\-struction, the number $N_t$ of spikes up to time $t$ is 
bounded by $\frac{1}{\delta_0}t$, for every $t\ge 0$). Under certain parameter configurations, 
spike trains can however be quite similar to what a Poisson process would show, in particular 
when very few spikes, all isolated ones, are observed over a long time interval (in contrast to 
this, see figures \ref{fig:1a} and \ref{fig:1}). In Definition~\ref{def:31}, the criterion 
\eqref{cond_**} corresponds to a nonrandomized Poisson test for the hypothesis 'unknown 
intensity is $\le \lambda_c$' versus  '$>\lambda_c$' which in a Poisson model would be 
uniformly most powerful for its level (\cite{Wi-85}, p.\ 210), and criteria \eqref{cond_+} and  
\eqref{cond_++} correspond to a Poisson goodness-of-fit test 
at the critical value $\lambda_c$: if $N$ were Poisson with intensity $\lambda_c$, the 
random 
variables $\Delta_{DF}(T_0,K)$ and $\Delta_{LT}(T_0,K)$ would exceed the critical values in 
\eqref{cond_+} and  \eqref{cond_++} in only $5$ out of $10^4$ cases, on average.

\begin{example}\label{exam:32}
We use  a simulation study to investigate the spiking behaviour of stochastic Hodgkin-Huxley 
neurons  \eqref{doublescript_X} with signal $\vartheta=4$, and  to illustrate the influence of 
the 
volatility $\sigma$ and the back-driving force $\tau$ in the Ornstein-Uhlenbeck process 
\eqref{def_X_HH} on the spiking behaviour.  
Recall from Section~\ref{sect:det_HH} that under random initial conditions, a deterministic 
neuron with signal $\vartheta=4$ would have its trajectories attracted to the stable equilibrium 
$\left( v^{\{\vartheta\}} ,  n^{\{\vartheta\}} , m^{\{\vartheta\}} , h^{\{\vartheta\}} \right)$ in 
\eqref{eq:equilibrium}, and trapped there.

Our simulations of trajectories for the stochastic neuron (Euler schemes with time step 
$0.001$) mimick stationary behaviour by omitting, after random initial conditions, a sufficiently 
long initial piece of trajectory (of length $1000$): its  terminal state will serve as starting point 
for a trajectory of total length $T_1=25000$, cut down into $K=100$ segments of length 
$T_0=250$, which we evaluate statistically.

Counting spikes on the $K=100$ path segments of length $T_0=250$ we define as in 
\eqref{K_Tnull_setting}
\[
T_1 := K  T_0 \quad,\quad 
\widetilde{\lambda} := \frac{1}{T_1} N_{T_1}  \quad , \qquad 
\xi_k :=  N_{kT_0} - N_{(k-1)T_0}    \quad,\quad 1\le k\le K    
\]
and use the criteria and critical values of Definition~\ref{def:31}. For every parameter 
configuration $(\vartheta\equiv 4,\tau,\sigma)$ which we consider, we do 10 simulation runs 
over total time $T_1=25000$. This is sufficient to obtain strong evidence -- see the tables in 
\ref{item:ex42_3}) below-- for the following: 
\begin{enumerate}[i)]
\item the probability $Q_\mu\left( Q(T_0,K) \right) =  Q^{(\vartheta= 
4,\tau,\sigma)}_\mu\left( 
Q(T_0,K) \right) $ that a simulation run of length $T_1$ will turn out to be quiet in the sense 
of Definition~\ref{def:31} is a function of  $\tau$ and $\sigma$; 
\item for fixed value of the volatility $\sigma$, sufficiently large values of the back-driving 
force 
$\tau$ (the meaning of 'large' depending on $\sigma$) make sure that the neuron will be quiet 
with probability close to $1$. 
\end{enumerate}
We explain this in more detail in the following steps \ref{item:ex42_1})--\ref{item:ex42_4}).  
\begin{enumerate}[1)]
\item\label{item:ex42_1} With $\lambda_c=0.0005$ from \eqref{crit_val_GoF}, the upper 
$5\%$-quantile 
\eqref{cond_**} of the Poisson law $\mathsf{P}_{\lambda_c T_1}$  equals  
\[
\bar{\mathsf{q}}\left({0.05},  \lambda_c T_1 \right)   = 19. 
\]
With $I=[0,5.5]$ and $\alpha_c=0.0005$ as in \eqref{crit_val_GoF}, we determine 
approximate quantiles for criteria \eqref{cond_+} and \eqref{cond_++} as follows. To 
approximate the law under $\mathsf{P}^K_{\lambda_c T_0}$ of the random variables  
\begin{equation}\label{cond_x}
	\int_I \left| \widehat{\mathsf{F}}_K(v) - \mathsf{F}_{\bar\xi_K}(v) \right| dv \quad,\quad 
	\int_I \left| \widehat{\phi_K}(v) - \phi_{\bar\xi_K}(v) \right| dv  
\end{equation}
we draw i.i.d.\ Poisson random variables $\xi_1,\ldots,\xi_K$ with parameter $\lambda_c 
T_0$ and calculate the integrals \eqref{cond_x} for these. After a large number of $4\cdot 
10^4$ 
replications, empirical distribution functions for the objects in \eqref{cond_x} are  sufficiently 
good to determine upper $\alpha_c$-quantiles approximately:  
\begin{equation}\label{the_approx_poisson_quantiles}
	\bar{\mathsf{q}}_{DF} ( \alpha_c ; \lambda_c T_0, K, I ) \approx 0.075    \quad,\quad  
	\bar{\mathsf{q}}_{LT} ( \alpha_c ; \lambda_c T_0, K, I ) \approx 0.15  . 
\end{equation}
Approximations \eqref{the_approx_poisson_quantiles} yield critical values   $c_{DF}$ and 
$c_{LT}$  for conditions \eqref{cond_+} and \eqref{cond_++}. 
\item\label{item:ex42_2} In every simulation run, with $\xi_1,\ldots,\xi_K$ and 
$\widetilde{\lambda}$ given by 
\eqref{K_Tnull_setting}, we have    
\[
\widetilde{\lambda} T_0 = \frac{ N_{T_1} }{ T_1 } T_0  = \frac{1}{K} N_{T_1} = \frac{1}{K}  
\sum_{j=1}^K \xi_j = \bar\xi_K   
\]
and calculate from  $\xi_1,\ldots,\xi_K$ the integrals 
\[
\Delta_{DF}(T_0,K)  :=   \int_I \left| \widehat{F_K}(v) - \mathsf{F}_{\widetilde{\lambda} T_0}(v) 
\right| dv 
\quad,\quad 
\Delta_{LT}(T_0,K)  :=   \int_I \left| \widehat{\psi_K}(v) - \phi_{\widetilde{\lambda} T_0}(v) 
\right| dv. 
\]
We check conditions \eqref{cond_+} and \eqref{cond_++} in combination with either 
\eqref{extremely_few} or \eqref{cond_**}. When the full set of conditions is satisfied, the 
simulation run is counted as quiet in the sense of Definition~\ref{def:31}. We repeat $10$ runs 
under every parameter configuration which we consider. 
\item\label{item:ex42_3} With signal $\vartheta=4$, we vary the volatility $\sigma$ and 
the back-driving force $\tau$. As 
a 
general feature, when $\sigma$ is fixed, spikes turn out to be rare  under 'large' values of  
$\tau$ whereas they are frequent under 'low' values of  $\tau$  (figure \ref{fig:1a} provides an 
illustration for the last case). In \ref{item:ex42_3_1})--\ref{item:ex42_3_3}) below, we report 
in more detail the outcome of the 
simulation study for selected values of $\sigma$ and $\tau$. 

\begin{enumerate}[i)]
\item\label{item:ex42_3_1} For signal $\vartheta=4$ and volatility $\sigma=2.5$, interesting 
$\tau$-values range 
between $2.0$ and $2.4$. Figure~\ref{fig:7} shows empirical distribution functions for the 
random variables $N_{T_1}$, $\Delta_{DF}(T_0,K)$ and $\Delta_{LT}(T_0,K)$ obtained from 
the $10$ simulation runs. Vertical dotted lines indicate the critical values in \eqref{cond_**}, 
\eqref{cond_+} and \eqref{cond_++} as specified in \ref{item:ex42_1}). 
The following percentages of runs turned out to be quiet in the sense of 
Definition~\ref{def:31}: 
\begin{center}
	\begin{tabular}{|l|l|l|l|l|}
		\hline
		$\tau=2.0$ & $\tau=2.1$ & $\tau=2.2$ &  $\tau=2.3$ & $\tau=2.4$ \\
		\hline
		0\% &  10\% & 60\% & 100\% & 100\% \\
		\hline
	\end{tabular}
\end{center}
Among runs which did not fulfill all requirements of Definition~\ref{def:31}, some  failed with 
respect to Poisson-goodness of fit \eqref{cond_+} or \eqref{cond_++}  while satisfying 
\eqref{cond_**}, and some runs failed to \eqref{cond_**} while satisfying  \eqref{cond_+} and 
\eqref{cond_++}. 
Figure \ref{fig:7} provides evidence that for $\vartheta$ and $\sigma$ fixed,  laws of 
$N_{T_1}$, $\Delta_{DF}(T_0,K)$ and $\Delta_{LT}(T_0,K)$ do depend on $\tau$. Figure 
\ref{fig:7}  suggests in addition that laws under $Q^{(\vartheta,\tau,\sigma)}_\mu$ of these 
three variables should be stochastically ordered in $\tau$, in the sense that larger values of 
the back-driving force tend (while reducing among all observed spikes the proportion of 
double 
or triple ones) to reduce the total number of spikes and to improve the quality of Poisson 
approximation. As an example, 
the 10 simulation runs (over total time $T_1=25000$) produced in average $37.4$ spikes 
under $\tau=2.0$, in contrast to $5.6$ in  average under $\tau=2.4$. 
\item\label{item:ex42_3_2} For signal $\vartheta=4$ and volatility $\sigma=1.5$, results of 
similar structure as 
described in \ref{item:ex42_3_1}) were observed, but now the interesting range of 
$\tau$-values is between 
$1.0$ and $1.4$. In the $10$ simulation runs, the following percentage turned out to be quiet 
in the sense of Definition~\ref{def:31}: 
\begin{center}
	\begin{tabular}{|l|l|l|l|l|l|}
		\hline
		$\tau=1.0$ & $\tau=1.1$ & $\tau=1.15$ & $\tau=1.2$ & $\tau=1.3$ & $\tau=1.4$ \\
		\hline
		0\% &   0\% &  30\% & 70\% & 80\% & 100\% \\
		\hline
	\end{tabular}
\end{center}
The empirical distribution functions in figure \ref{fig:8} illustrate the dependence of the laws of 
all three variables $N_{T_1}$, $\Delta_{DF}(T_0,K)$ and $\Delta_{LT}(T_0,K)$ on the 
back-driving force $\tau$, and provide a strong hint that laws under 
$Q^{(\vartheta,\tau,\sigma)}_\mu$ of the three variables should be stochastically ordered in 
the sense of decreasing values of $\tau$. Most striking example, the $10$ simulation runs 
(with $T_1=25000$) produced an average of $51.1$ spikes under $\tau=1.0$, in strong 
contrast to an average of only $2.5$ under $\tau=1.4$.

\item \label{item:ex42_3_3}For signal $\vartheta=4$ and volatility $\sigma=1.0$, we observe 
the same features as in \ref{item:ex42_3_1})
and \ref{item:ex42_3_2}). The interesting range of $\tau$-values is now between $0.5$ 
(frequent spikes) and 
$0.75$ (few spikes), and the following percentages of runs turned out to be quiet  in the sense 
of Definition~\ref{def:31}: 
\begin{center}
	\begin{tabular}{|l|l|l|l|l|l|}
		\hline
		$\tau=0.5$ & $\tau=0.55$ & $\tau=0.6$ & $\tau=0.65$ &  $\tau=0.7$ & $\tau=0.75$ \\
		\hline 
		0\% &   20\% &  50\% & 90\% & 100\% & 100\% \\
		\hline
	\end{tabular}
\end{center}
\end{enumerate}

\item\label{item:ex42_4} We sum up as follows: when the signal is $\vartheta=4$, 
the schemes in \ref{item:ex42_3_1})--\ref{item:ex42_3_3}) above prove the dependence of 
$Q^{(\vartheta,\tau,\sigma)}_\mu\left( 
Q(T_0,K) \right)$ on the volatility $\sigma>0$ and the back-driving force $\tau>0$. 
Stronger values of $\tau$ tend to reduce the total number of spikes and to improve the quality 
of Poisson approximation. 
For actual determination of probabilities of events $Q(T_0,K)\in\mathcal{G}_{T_1}$ in 
stationary regime, 
we would of course need much more than the 10 simulation runs (up to time  $T_1=25000$, 
under every parameter configuration) which we have done. So the above tables above can 
give 
only poor approximations so far. 
\end{enumerate}
\end{example}

\begin{figure}[h]
	\centering
	\includegraphics[width=0.95\linewidth]{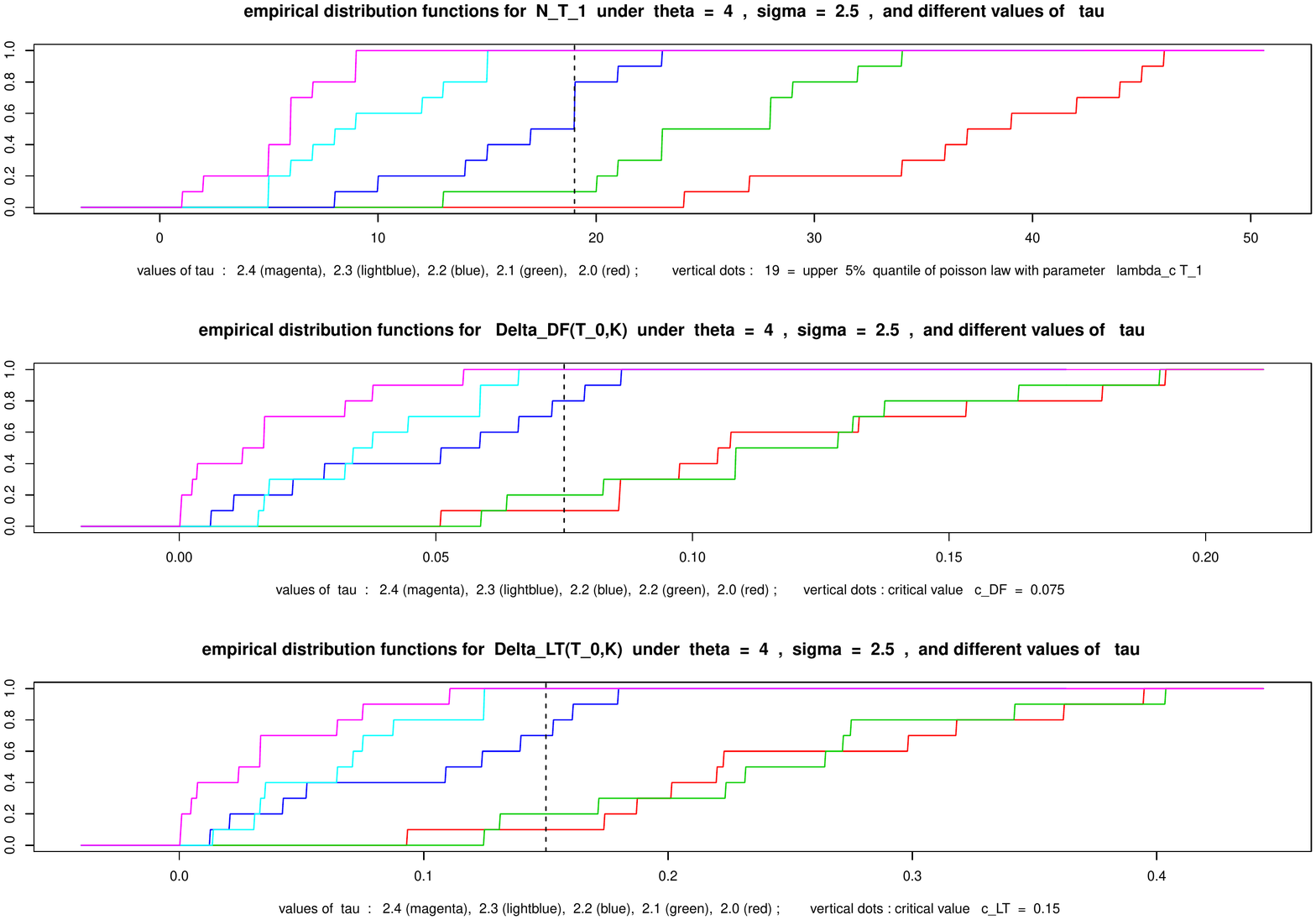} 
	\caption{
		In the stochastic Hodgkin-Huxley model with signal $\vartheta=4$, volatility 
		$\sigma=2.5$, and values of the back-driving force $\tau$ varying between  $2.0$ and 
		$2.4$, we show empirical distribution functions for the laws of the random variables 
		$N_{T_1}$, $\Delta_{DF}(T_0,K)$, $\Delta_{LT}(T_0,K)$ under 
		$(\vartheta,\sigma,\tau)$.  
		These are based on the values which have been observed in the $10$ simulation runs 
		described in \ref{item:ex42_3_1}) of example~\ref{exam:32}, on a time interval of length 
		$T_1=25000$ 
		divided into 
		$K=100$ segments of length $T_0=250$.  
		The graphics suggest stochastic ordering of the laws under 
		$Q^{(\vartheta,\tau,\sigma)}_\mu$ for all three variables, most clearly visible in case of 
		the total number $N_{T_1}$ of spikes, in the sense of decreasing values of $\tau$.  
	}
	\label{fig:7}
\end{figure}

\begin{figure}[h]
	\centering
	\includegraphics[width=0.95\linewidth]{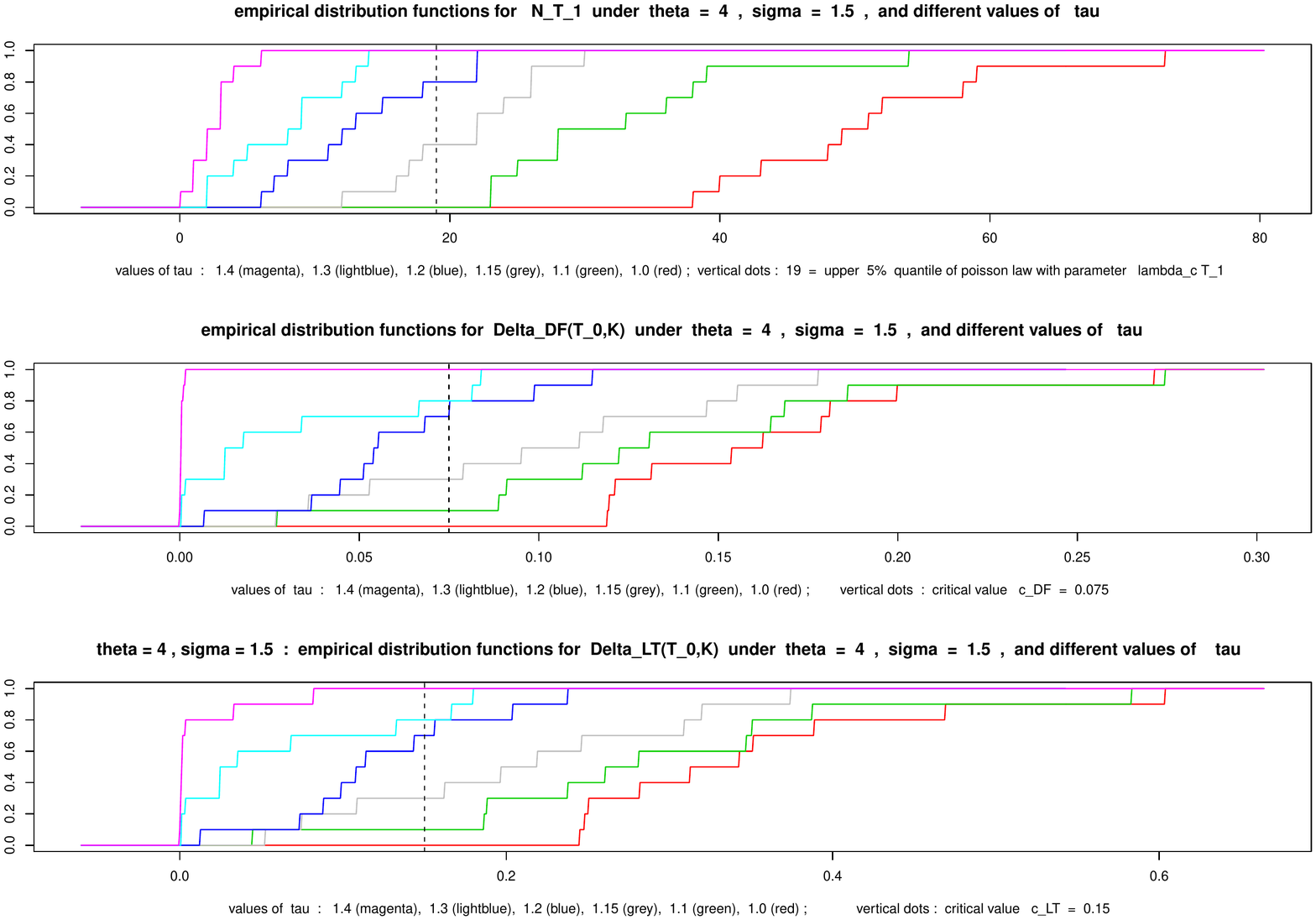} 
	\caption{
		In the stochastic Hodgkin-Huxley model with signal $\vartheta=4$, volatility 
		$\sigma=1.5$, and values of the back-driving force $\tau$ varying between  $1.0$ and 
		$1.4$, we show empirical distribution functions  for the laws of the random variables 
		$N_{T_1}$, $\Delta_{DF}(T_0,K)$, $\Delta_{LT}(T_0,K)$ under 
		$(\vartheta,\tau,\sigma)$. 
		They are based  on the values which have been observed in the $10$ simulation runs 
		described in \ref{item:ex42_3_2}) of Example~\ref{exam:32}, on a time interval of length 
		$T_1=25000$ 
		divided into 
		$K=100$ segments of length $T_0=250$.  
		The graphics suggest strongly that laws under $Q^{(\vartheta,\tau,\sigma)}_\mu$ for 
		all three variables should be  stochastically ordered in the sense of decreasing values of 
		$\tau$. 
	} 
	\label{fig:8}
\end{figure}

\section{Regular spiking of stochastic neurons } \label{sect:regular_spiking}

In deterministic Hodgkin-Huxley neurons, regular spiking --in the sense that trajectories are 
attracted towards a stable orbit-- depends only on the value of the signal $\vartheta$: with 
notations of Section~\ref{sect:det_HH},   this is the case $\vartheta > \sup( \mathbbm{I}_{\rm 
	bs})$  where we exclude, as in Section~\ref{sect:det_HH}, unrealistically large values of the 
signal.  

For a stochastic Hodgkin-Huxley neuron, given Proposition~\ref{prop:22} or 
Theorem~\ref{theo:25}, we 
shall define regular spiking as an event where up to some  sufficiently large time $T_1$, the 
pattern of observed spike times is sufficiently close to a regularly spaced grid whose step size 
is  the median 
\begin{equation}\label{median_regular_ISI}
	\Delta(T_1) := {\rm median}\left( \tau_2 - \tau_1, \ldots , \tau_{N_{T_1}} - \tau_{ 
		N_{T_1}-1 }   \right) 
\end{equation}
of the interspike times. Simulations indicate that the probability of such events in 
$\mathcal{G}_{T_1}$  in stationary regime depends on the triplet of parameters 
$(\vartheta,\tau,\sigma)$: for signal $\vartheta > \sup( \mathbbm{I}_{\mathrm{bs}})$, we 
can 
expect regular spiking in the sense of Definition~\ref{def:41} with probability close to $1$ 
whenever 
back-driving force $\tau$ in the Ornstein-Uhlenbeck process \eqref{def_X_HH} is --in relation 
to the value $\sigma$ of the volatility-- large enough.

Fix $T_1<\infty$ and assume that a sufficiently large number of interspike times  
\begin{equation}\label{regular_ISI}
	\tau_2 - \tau_1, \ldots , \tau_{N_{T_1}} - \tau_{ N_{T_1}-1 }    
\end{equation}
has been observed up to time $T_1$. Write $\widehat{H}_{T_1}$ for the empirical distribution 
function of the data set \eqref{regular_ISI}. Write 
\[
\Delta(T_1) = \inf\{ v>0 : \widehat{H}_{T_1}(v)\ge 0.5 \} 
\]
for the median and
\[
\mathtt{d} (\alpha,T_1) := \inf\{v>0 : \widehat{H}_{T_1}(v) \ge 1-\alpha \} - \inf\{ v>0 : 
\widehat{H}_{T_1}(v)\ge 
\alpha\}  
\]
for the distance between upper and lower $\alpha$-quantiles in the data set  
\eqref{regular_ISI}, $0<\alpha<\frac{1}{2}$; 
relating quantile distances to the median we shall consider ratios 
\begin{equation}\label{ratio_regular_ISI}
	\mathtt{r}(\alpha,T_1) := \frac{\mathtt{d}(\alpha,T_1)}{\Delta(T_1)}  . 
\end{equation}

The next Definition builds on weak convergence of empirical distributions for the  interspike 
times, as time goes to infinity, in application of Proposition~\ref{prop:22}: in the long run 
under $(\vartheta,\tau,\sigma)$, 'typical' interspike times are distributed according to   
$H^{(\vartheta,\tau,\sigma)}$.    
If we have no grasp on the limiting  object  itself, presence of noise in the system  
\eqref{stoch_HH} --as illustrated by figures such as  \ref{fig:1a}, \ref{fig:1} or \ref{fig:3} or by 
detailed representations of the system evolving on 'orbits' -- strongly suggests that 
$H^{(\vartheta,\tau,\sigma)}$ in
Proposition~\ref{prop:22}~\ref{item:22b}) has to be strictly increasing and continuous on 
some interval of 
support. In fact, flats in the limit distribution function seem impossible under noise --this 
would imply existence of pairs $0<x_1<x_2<\infty$ with $0 < H^{(\vartheta,\tau,\sigma)}(x_1) 
= H^{(\vartheta,\tau,\sigma)}(x_2) < 1$ and thus non-existence of interspike times of length 
between $x_1$ and $x_2$ in the long run under $(\vartheta,\tau,\sigma)$-- as well as point 
masses. But then,  arbitrary quantiles of the empirical distribution functions  (\cite{WMF-95} 
p.\ 71, without exceptional set) should converge to those of the limit distribution function, as a 
consequence of  Proposition~\ref{prop:22}~\ref{item:22b}) . See also 
Proposition~\ref{prop:73} in the 
Appendix 
Section~\ref{conjectures}.

\begin{definition}\label{def:41}
Consider  a stochastic Hodgkin-Huxley neuron \eqref{doublescript_X} 
under $(\vartheta,\tau,\sigma)$. With notations   
\eqref{median_regular_ISI}--\eqref{ratio_regular_ISI} and for $T_1$ large enough,  define 
$R(T_1)$ as the event in $\mathcal{G}_{T_1}$ on which 
\begin{equation}\label{zusatzcond_4.1} 
	\left| \frac{ N_{T_1} \Delta(T_1) }{T_1} - 1 \right| \le 0.05    \quad\mbox{ and }\quad 
	N_{T_1}>20 
\end{equation}
holds together with  
\begin{equation}\label{cond_4.1}
	\mathtt{r}(0.05, T_1) \le 0.3,\quad 
	\mathtt{r}(0.1, T_1)  \le 0.2 ,\quad
	\mathtt{r}(0.25, T_1)   \le 0.1.  
\end{equation}
On the event $R(T_1)\in\mathcal{G}_{T_1}$ we call the stochastic neuron $\mathbb{X}$ 
\textbf{regularly spiking}. 
\end{definition}

Condition \eqref{cond_4.1} requires that quantiles in the data set \eqref{regular_ISI} of 
interspike times are close to the median, but does not rule out (as illustrated by 
Figure~\ref{fig:1}) that up to time $T_1$, few long spikeless periods alternate with long groups 
of 
regularly spaced spikes. 
This is why  \eqref{zusatzcond_4.1} requires  regular spacing on at least $95\%$ of the time 
interval on which the membrane potential is observed.

\begin{example}\label{ex:42}
	As in classical statistics of i.i.d.\ observations, 
	we can approximate $Q^{(\vartheta,\tau,\sigma)}_\mu\left( R(T_1) \right)$ from 
	independent replications of $\mathbb{X}^{(\vartheta,\sigma,\tau)}$ in stationary regime 
	over time intervals of length $T_1$. We consider signal $\vartheta=10$ (for which the 
	deterministic process would evolve along a stable orbit, see section \ref{sect:det_HH}) 
	together with different values of $\tau$ and $\sigma$. 
	
	We use Euler schemes of step size $0.001$. In order to mimick stationary regime, in every 
	simulation run, we cast away an initial piece of trajectory (of length $100$) under randomly 
	chosen initial conditions (uniformly on  $(-12,120){\times}(0,1)^3$ for $(V,n,m,h)$ and 
	according to the stationary law for $X$; the output $U$ starts at $0$), conserve the final 
	state of this initial piece of trajectory as starting point for the simulation of interest which 
	then covers a time interval of length $T_1=500$. This second piece of trajectory is used for 
	inference.     
	
	Thus, for $\vartheta=10$ and $T_1=500$, we simulate  $20$ runs under every parameter 
	configuration. The scheme below gives the proportion of runs where regular spiking in the 
	sense of Definition~\ref{def:41} was observed.  
	\begin{center}
		\begin{tabular}{|l|l|l|l|l|l|}
			\hline
			~  &  $\tau=0.1$  &  $\tau=0.5$  &  $\tau=1$  &  $\tau=2.5$  &  $\tau=5$  \\
			\hline
			$\sigma=1$ & 100\% & 100\%  & 100\%  & $^*$100\%   & $^*$100\%   \\
			\hline
			$\sigma=1.5$ & 55\% & 90\%  & 100\%  & 100\%  & $^*$100\%  \\
			\hline
			$\sigma=2.5$ & 15\% & 45\%  & 75\% & 100\% & 100\%  \\
			\hline
			$\sigma=5$ & 0\% & 5\%  & 20\% & 95\% & 100\%  \\
			\hline
		\end{tabular} 
	\end{center}
	We deduce that in stationary regime, the event $R(T_1)\in\mathcal{G}(T_1)$ has probability 
	close to $1$ for well-chosen pairs $(\tau,\sigma)$: either, depending on $\sigma$, 
	the back-driving force $\tau$ has to be strong enough, or, depending on $\tau$, the 
	volatility 
	$\sigma$ has to be small enough. 
	Asterisk $^*$ distinguishes parameter configurations under which observed ratios 
	$\mathtt{r}(0.05,T_1)$ in \eqref{cond_4.1} turned out --on average over the $20$ 
	simulation 
	runs-- 
	to be strictly smaller than $0.05$, whereas the median  $\Delta(T_1)$ was located between 
	$14.3$ and $14.4$. In this sense, an upper right triangle of parameter values shows up in 
	the scheme    
	where quantiles of the data set \eqref{regular_ISI} concentrate sharply at the median, and 
	empirical distribution functions for the observed values of  
	\begin{equation}\label{triplet_ratios}
		\mathtt{r}( 0.05 , T_1 ) \quad,\quad   \mathtt{r}( 0.1 , T_1 )  \quad,\quad 
		\mathtt{r}( 0.25 , 
		T_1 ) 
	\end{equation}
	in the $20$ simulation runs under $\vartheta=10$ (not shown) look very much as a Dirac 
	mass at $\Delta(T_1)$ under 'some small random perturbation'. Figure \ref{fig:9} below 
	shows empirical distribution functions for observed values of ratios   \eqref{triplet_ratios}  in 
	case of rather high volatility  $\sigma=2.5$ for all values $\tau \in \{ 0.1 , 0.5 , 1.0 , 2.5 , 
	5.0 \}$ in the scheme under $\vartheta=10$. 
	The graphics  indicate  that   in stationary regime, laws of variables \eqref{triplet_ratios} 
	seem to be stochastically ordered in the  back-driving force $\tau$, in the sense that 
	increasing values of $\tau$ tend to push  quantiles $\mathtt{q}(\alpha,T_1)$ closer to the 
	median 
	$\Delta(T_1)$.
\end{example}
By \eqref{def_output_process}--\eqref{output_properties}, the output $U$ of a stochastic 
Hodgkin-Huxley neuron \eqref{doublescript_X} fluctuates between 'typical values' of 
$U_{\tau_\ell}$ (local maxima) and   $U_{(\tau_{\ell+1})^-}=U_{\tau_\ell} e^{ -c_1 
	(\tau_{\ell+1} - \tau_\ell) }$  (local minima) as $\ell\to\infty$. In general, typical values refers 
	to 
the approximations and limit distributions of Proposition~\ref{prop:26}. If however the neuron 
is regularly spiking, 'typical values' takes a much sharper sense:   simulations suggest that  on 
events $R(T_1)$ as defined in \ref{def:41}, for  $T_1$ large enough, functions of 
$\Delta(T_1)$ 
\begin{equation}\label{benchmarks_T1}
	\sum_{j\ge 0} e^{ - c_1 \Delta(T_1) j } = \frac{ 1 }{ 1 -  e^{ - c_1 \Delta(T_1) } } 
	\quad,\quad 
	\sum_{j\ge 0} e^{ - c_1 \Delta(T_1) (j+1) } = \frac{ e^{ - c_1 \Delta(T_1) } }{ 1 -  e^{ - 
	c_1
			\Delta(T_1) } } 
\end{equation}
provide benchmarks which allow to predict where pairs  $U_{\tau_\ell}$, 
$U_{(\tau_{\ell+1})^-}$ tend to cluster  in the long run, hence predict an interval on which the 
output process tends to concentrate a predominant part of future occupation time.  Figure 
\ref{fig:3} provides an illustration.

We shall discuss in an appendix section \ref{conjectures} in which sense 
\eqref{benchmarks_T1} is expected to provide good approximations to future values of the 
output process, on events $R(T_1)$ when $T_1$ is large.

In numerous simulations with large values of signal $\vartheta$ and large observation time 
$T_1$, whenever the stochastic neuron turned out to be regularly spiking in the sense of 
Definition~\ref{def:41},  benchmarks \eqref{benchmarks_T1} predicted well the range of 
oscillations of 
the output process once time was large enough.

\begin{figure}[h]
	\centering
	\includegraphics[width=0.95\linewidth]{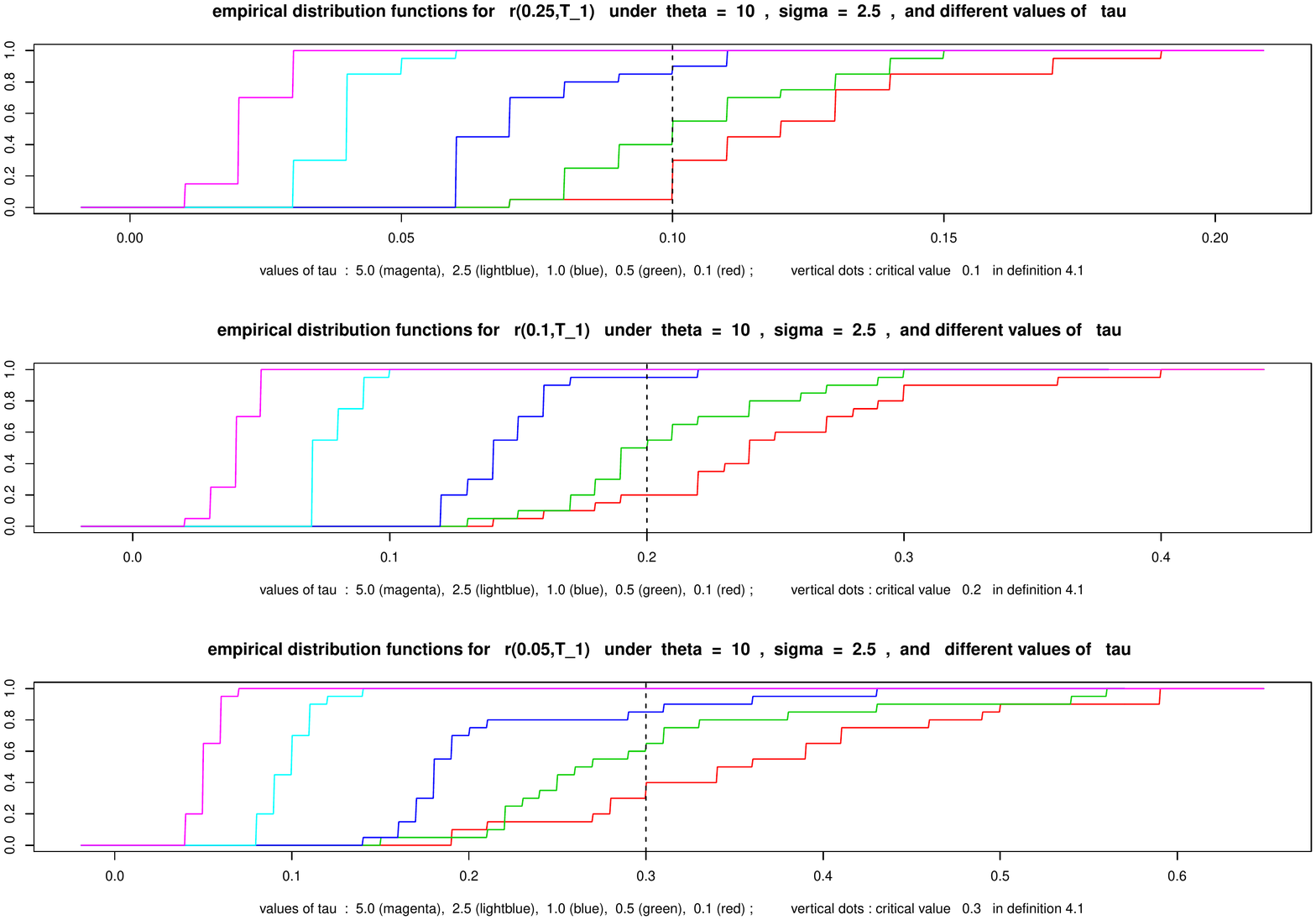} 
	\caption{
		In the stochastic Hodgkin-Huxley model with signal $\vartheta=10$, volatility 
		$\sigma=2.5$, and values of the back-driving force $\tau$ varying between $0.1$ and 
		$5.0$, we show empirical distribution functions for the laws of the random variables 
		$\mathtt{r}(0.25,T_1)$,  $\mathtt{r}(0.1,T_1)$,  $\mathtt{r}(0.05,T_1)$ under 
		$(\vartheta,\sigma,\tau)$,  in stationary regime and with $T_1=500$. 
		The empirical distribution functions are based on the values which have been observed in 
		the 
		$20$ simulation runs  described in Example~\ref{ex:42}. 
		The graphics suggest that laws of all three random variables under 
		$Q^{(\vartheta,\tau,\sigma)}_\mu$ should be stochastically ordered in $\tau$, in the 
		sense that increasing values of $\tau$ improve remarkably the concentration of interspike 
		times around their median. 
		The median $\Delta(T_1)$ itself (as well as the number of spikes in the observation 
		interval) does not change much with $\tau$: averaged over the 20 runs we obtained 
		$14.19$ for $\tau=0.1$, and  $14.33$ for $\tau=5.0$.
	}
	\label{fig:9}
\end{figure}

\begin{figure}[h]
	\centering
	\includegraphics[width=0.95\linewidth]{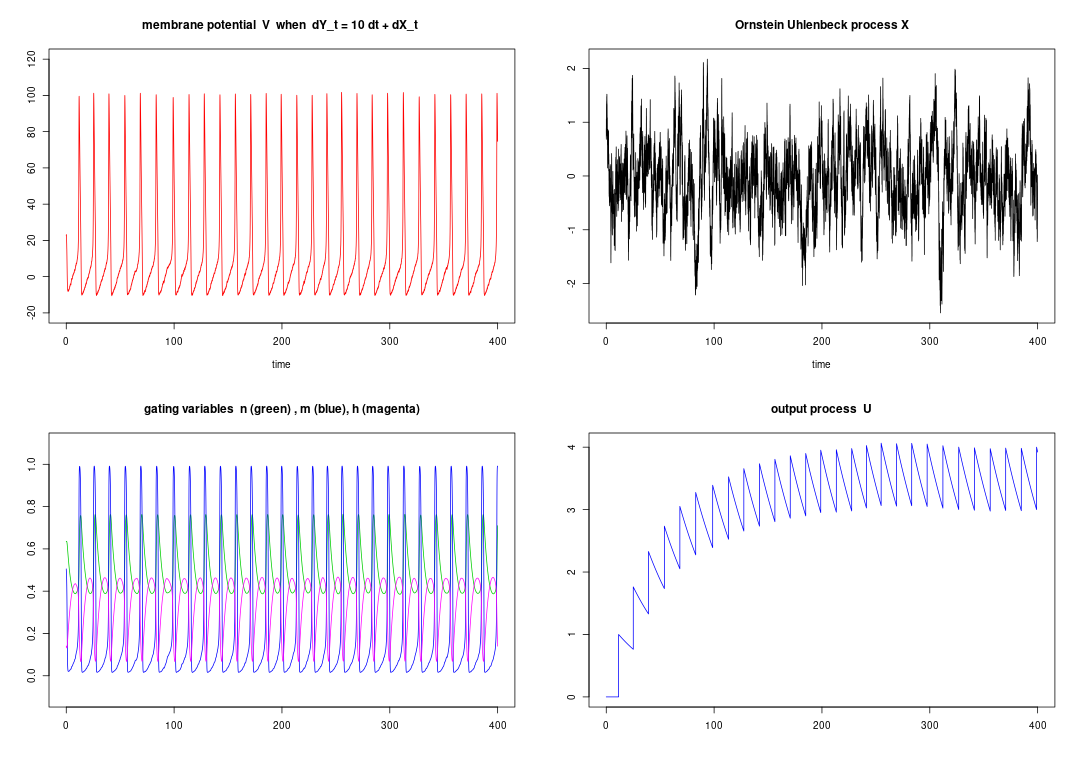} 
	\caption{Simulated trajectory of a stochastic Hodgkin-Huxley neuron 
	$\mathbb{X}^{(\vartheta,\tau,\sigma)}$. The signal is $\vartheta=10$. The parameters for 
	the OU process $X$ are $\tau=0.7$ and $\sigma=0.83666$. The decay parameter in the 
	output process $U$ is $c_1 = 0.02$. The simulation was done using an Euler scheme with 
	equidistant steps 0.001. We start with $U_0=0$ and random initial conditions for 
	$(V,n,m,h,X)$. 
		In this simulation, $27$ spikes occur up to time $T_1=400$. The median of the $26$ 
		interspike times is $\Delta(T_1)=14.4$,  
		lower resp.\ upper $25\%$-quantiles are at  $14.07$ resp.\  $14.69$,  
		the minimum is $13.87$ and the maximum $15.23$.   
		Longer and shorter interspike times seem to alternate at random. 
		Benchmarks \eqref{benchmarks_T1} take the value $3.99$ 
		for local maxima of the output process in the long run, and  $2.99$ for local minima. The 
		interval $(2.99,3.99)$ is in good fit with the range of oscillations of the output process 
		$U$ on the second half of the time interval of observation.   
	}
	\label{fig:3}
\end{figure}

\section{Circuits of stochastic Hodgkin-Huxley neurons}\label{sect:circuits}

This section describes circuits of interacting stochastic Hodgkin-Huxley neurons where 
activity shows up in form of blocks of spiking neurons performing slow and rhythmic 
oscillation 
around the circuit. This self-organized rhythmic behaviour of activity patterns seems to be 
persistent in the long run. We are however unable to prove persistence --however strongly 
suggested by simulations-- and restrict this section to a detailed description of the 
construction and to some motivating remarks. 
Our construction relies on the notions of 'quiet behaviour' (Definition~\ref{def:31}) and of 
'regular spiking' (Definition~\ref{def:41})in order to define the interactions between the 
neurons in the 
circuit.

It  seems admitted that in networks of biological neurons \cite{Izh-07,ET-10}, information 
transfer happens in form of excitation and inhibition at a large number of synapses, a by far 
larger part of the synapses in the network being excitatory.  
In our circuit of interacting stochastic Hodgkin-Huxley neurons, we impose a block structure 
of 
the following type:  information transfer from one neuron to the next along the circuit will be 
excitatory as long as we remain inside the same block, and will be inhibitory when we pass 
from 
the last neuron in a block to the first neuron in its successor block. In this block-wise 
construction 
of the circuit, as  a consequence of  excitation and inhibition,  self-organized patterns of 
oscillation show up quite rapidly. Spiking activity is propagating from block to block around 
the 
circuit:  while some blocks are regularly spiking and in this sense active, others are quiet at the 
same time, and  at certain times, blocks flip from active to quiet, and back from quiet to active. 
In this way,  block-wise activity patterns  arise and perform a slow rotation along the circuit. 
This rotational movement seems to be persistent. However, nothing being proved so far, we 
only give the construction.

\begin{modstep}[Selecting suitable parameter values]\label{subsec:51} 
	~
With reference to the 
bistability interval 
$\mathbbm{I}_{\mathrm{bs}}$ of the deterministic case in Section~\ref{sect:det_HH}, we fix 
parameter values 
\begin{equation}\label{def:param_values}
	0 < \vartheta_1 := \lfloor \inf( \mathbbm{I}_{\mathrm{bs}} ) - 1 \rfloor 
	,\quad 
	\vartheta_2 := \lceil \sup( \mathbbm{I}_{\mathrm{bs}} ) + 1 \rceil 
	,\quad 
	0<\tau<\infty \quad,\quad 0<\sigma<\infty 
\end{equation}
such that a single stochastic Hodgkin-Huxley neuron \eqref{doublescript_X} in stationary 
regime tends to be  
\begin{align*}
\mbox{ quiet under } (\vartheta_1,\tau,\sigma) \mbox{  in the sense of Definition~\ref{def:31} }, 
\\
\mbox{ regularly spiking under } (\vartheta_2,\tau,\sigma) \mbox{   in the sense of 
Definition~\ref{def:41}}, 
\end{align*}
where we require that the back-driving force $\tau$ and the volatility $\sigma$ be the same in 
both 
cases.

Our choices in modelization step~\ref{subsec:51} require more explanation. 
First, in stationary regime and for $T_0$ and $K$ sufficiently large in the sense of 
Definition~\ref{def:31} and Example~\ref{exam:32}, we have to identify pairs $(\tau,\sigma)$ 
such that the 
probability to find the single stochastic Hodgkin-Huxley neuron \eqref{doublescript_X} under 
$(\vartheta_1,\tau,\sigma)$ in the event $Q(T_0,K)\in\mathcal{G}_{T_1}$ is close to $1$. 
Second, in stationary regime and for $T_1$ as in Definition~\ref{def:41} and 
Example~\ref{ex:42} (which 
represents a different choice of $T_1$), we have to identify pairs $(\tau,\sigma)$ such that the 
probability to find a single stochastic Hodgkin-Huxley neuron \eqref{doublescript_X} under 
$(\vartheta_2,\tau,\sigma)$ in the event $R(T_1)\in\mathcal{G}_{T_1}$ is close to~$1$. 
Third, we have to select \textbf{one} pair $(\tau,\sigma)$ which meets both requirements. This 
is possible. 
\end{modstep}
~

\begin{example}\label{ex:52}
As an example, with values 
\begin{equation}\label{choice_theta1_theta2_circuit}
	\vartheta_1 =  4  \quad,\quad    \vartheta_2 = 10
\end{equation}
satisfying \eqref{def:param_values} according to the numerical considerations of 
Section~\ref{sect:det_HH}, the simulations in Examples~\ref{exam:32} and \ref{ex:42} indicate 
that a combination of the volatility and the back-driving force of the form 
\[
\sigma=1.0 \quad\mbox{and}\quad \tau\ge 0.7 \quad,\quad\sigma=1.5 \quad\mbox{and}\quad 
\tau\ge 1.4 
\quad,\quad  \sigma=2.5 \quad\mbox{and}\quad  \tau\ge 2.5   
\]
meet the requirements of  modelization step~\ref{subsec:51}. See the schemes in step 
\ref{item:ex42_3}) of 
\ref{exam:32}, and 
the scheme in Example~\ref{ex:42}. 

For the stochastic neuron under $(\vartheta_2,\tau,\sigma)$ --regularly spiking with 
probability close to $1$ up to time $T_1$--  we fix a prediction  
$\Delta^* =  \Delta^*(\vartheta_2,\tau,\sigma)$  
for the median of interspike times in stationary regime in the long run: combining 
modelization step~\ref{subsec:51} and 
Example~\ref{ex:42}, with $T_1$,  $R(T_1)$, $\Delta(T_1)$ from Example~\ref{ex:42},  we 
define   
\begin{equation}\label{def_nustar}
	\Delta^* := \Delta(T_1)  
\end{equation}
which is the median of the interspike times observed up to time $T_1$. 
Every choice of a decay parameter $c_1$ in view of a construction 
\eqref{def_output_process} of an output process $U$  then associates to $\Delta^*$ an 
interval  \eqref{benchmarks_T1} 
\begin{equation}\label{nustar+ustars}
	(u^*_1 , u^*_2) \quad,\quad 
	u^*_2 := \sum_{j=0}^\infty e^{ - c_1 j \Delta^*}   
	\quad,\quad u^*_1 :=  \sum_{j=0}^\infty e^{- c_1 (j+1) \Delta^*}  
\end{equation}
on which we expect the output process $U$ to accumulate a large amount of occupation time 
in the long run. We wish to  scale the shape of output processes for regularly spiking neurons 
more or less independently of the parameters $(\vartheta_2,\tau,\sigma)$ and thus  of 
$\Delta^*$. 
\end{example}

\begin{modstep}[Calibration of the decay parameter for output 
	processes]\label{subsection:53} From $\Delta^*$ in 
\eqref{def_nustar} define 
\[
c^* :=  -\log\left(\frac{3}{4}\right) / \Delta^* ;  
\]
then, by an 'adapted' choice of the decay parameter $c_1>0$ in \eqref{def_output_process} , 
we 
consider  
intervals \eqref{nustar+ustars} which approximately do not depend on the parameters:  
\begin{equation}\label{def:approxcstar}
	c_1 \approx  c^* \quad,\quad u^*_2 = ( 1 - e^{ - c_1  \Delta^*} )^{-1}\approx 
	( 1 - e^{ - c^*  \Delta^*} )^{-1}  = 4 \quad,\quad  u^*_1 \approx 3 .  
\end{equation}
\end{modstep}
~
\begin{modstep}[Choice of transmission functions]\label{subsection:54}  Select some 
smooth 
	and strictly increasing 
function $\Psi^* : \mathbb{R} \to [0,1]$ with the properties  
\begin{equation}\label{cond_psinod}
	\lim_{v\to -\infty} \Psi^*(v) = 0 \quad,\quad
	\Psi^*(1) < 0.025 \quad,\quad \Psi^*(u^*_1) > 0.975  
	\quad,\quad  \lim_{v\to \infty} \Psi^*(v) = 1  
\end{equation}
for $c_1$ and $(u^*_1,u^*_2)$ selected in \eqref{def:approxcstar}. 
For $\vartheta_1 < \vartheta_2$ determined in modelization step~\ref{subsec:51}, use 
$\Psi^*$ to 
define a pair 
of 
transmission functions  
\begin{equation}\label{def_transition_functions}
	\left\{ \begin{array}{lll}
		\Psi_{\mathrm{exc}}(x) &:=& \vartheta_1 + (\vartheta_2-\vartheta_1)\Psi^*(x)  \\
		\Psi_{\mathrm{inh}}(x) &:=& \vartheta_2  -  (\vartheta_2-\vartheta_1)\Psi^*(x) 
	\end{array} \right.
	\quad,\quad x\in\mathbb{R}. 
\end{equation}
The first function in  \eqref{def_transition_functions} will be used to model excitation, the 
second inhibition.

As an example, using well known properties of the standard normal distribution function 
$\Phi$ and its quantiles, a choice  
\[
\Psi^*(x) :=  \Phi\left(\frac{ x - \frac{1+u^*_1}{2} }{ \frac{u^*_1-1}{6}}\right)  
\]
will satisfy \eqref{cond_psinod}. The transmission functions in 
\eqref{def_transition_functions}, excitatory or inhibitory, serve as a key tool to model 
information transfer between neurons in the circuit under construction. 
\end{modstep}
~
\begin{modstep}[Construction of the circuit]\label{subsection:55} Fix an integer $M\ge 3$ 
	which is odd, and some integer 
$L \ge 4 $. 
We shall construct a circuit of $N := ML$ neurons  
\begin{equation}\label{circuit_N_neurons_in_M_blocs}
	\mathcal{N}^{(1)} ,\ldots,  \mathcal{N}^{(N)}   
\end{equation}
where we count neurons around the circuit modulo $N$: in particular, $\mathcal{N}^{(0)}$ 
and $\mathcal{N}^{(N)}$ are different names for the same neuron in the circuit,  
$\mathcal{N}^{(N+1)}$ is $\mathcal{N}^{(1)}$, and so on. 
We arrange neurons along the circuit \eqref{circuit_N_neurons_in_M_blocs}  in $M$ blocks 
\begin{equation}\label{M_blocs}
	\{1,\ldots,L\}  ,  \{L{+}1,\ldots,2L\}  ,  \ldots  , \{(M-1)L + 1, \ldots,N\}  
\end{equation}
each of which contains $L$ neurons. Subsets of indices 
\begin{equation}\label{def:inh_exc_subsets}
	I_{\mathrm{inh}} := \{ 1 , L{+}1 , \ldots ,  (M-1)L + 1 \} \quad,\quad I_{\mathrm{exc}} 
	:= \{ 1 , 
	\ldots , 
	N \} \setminus  I_{\mathrm{inh}}    
\end{equation}
will be used to distinguish neurons $i\in I_{\mathrm{inh}}$ which  occupy the first position in 
their 
block (i.e.: $i$ equals $1$ modulo $L$) from neurons $i\in I_{\mathrm{exc}}$ which have 
their 
predecessor in the same block. We emphasize that the number $M$ of blocks in 
\eqref{M_blocs} has to be \textbf{odd}. 

In the circuit \eqref{circuit_N_neurons_in_M_blocs} with its block structure 
\eqref{M_blocs}--\eqref{def:inh_exc_subsets}
--where the successor of neuron $\mathcal{N}^{(N)}$ is $\mathcal{N}^{(1)}$ and the 
predecessor of neuron $\mathcal{N}^{(1)}$ is $\mathcal{N}^{(N)}$ , in the sense of the 
circuit--  
neurons $\mathcal{N}^{(i)}$, $i\in I_{\mathrm{exc}}$, will be excited by their predecessor, 
and 
neurons $\mathcal{N}^{(i)}$, $i\in I_{\mathrm{inh}}$, will be inhibited by their predecessor. 
So 
transfer inside blocks will always  be excitatory; from the last neuron in a block to the first 
neuron 
in the following block, transfer will be inhibitory.  
\begin{enumerate}[a)]
\item\label{item:20221228a} For the pair $(\tau,\sigma)$ which has been selected in 
modelization step~\ref{subsec:51}, 
prepare $N = M L$ 
independent Ornstein-Uhlenbeck processes $X^{(i)}$, strong solutions to equations  
\[
dX^{(i)}_t = -\tau X^{(i)}_t dt + \sigma dW^{(i)}_t \quad,\quad 1\le i\le N 
\]
driven by independent Brownian motions $W^{(i)}$. We stress that by choice in 
modelization step~\ref{subsec:51},  
the back-driving force and the volatility are the same for all processes $X^{(i)}$, $1\le i\le N$. 
\item\label{item:20221228b} With stochastic processes $A^{(i)}= ( A^{(i)}_t )_{t\ge 0}$ 
designed to model 
interaction 
and to be explained in~\ref{item:20221228d}) below, we define the $i$-th neuron 
$\mathcal{N}^{(i)}$ in the 
circuit, $1\le i\le N$, as a stochastic process 
\[
\mathcal{N}^{(i)} := 
\left( V^{(i)}, n^{(i)}, m^{(i)}, h^{(i)}, X^{(i)} \right)
\]
governed by a stochastic Hodgkin-Huxley equation of form 
\begin{equation}\label{circuit_single_HH}
	\left\{\begin{array}{lll}
		dV^{(i)}_t &= & A^{(i)}_t dt + dX^{(i)}_t  -  F(V^{(i)}_t, n^{(i)}_t, m^{(i)}_t, h^{(i)}_t) 
		dt \\
		dn^{(i)}_t &= &[\alpha_n(V^{(i)}_t)(1-n^{(i)}_t)  -  \beta_n(V^{(i)}_t) n^{(i)}_t] dt \\
		dm^{(i)}_t &= &[\alpha_m(V^{(i)}_t)(1-m^{(i)}_t)  -  \beta_m(V^{(i)}_t) m^{(i)}_t] dt 
		\\
		dh^{(i)}_t &= &[\alpha_h(V^{(i)}_t)(1-h^{(i)}_t)  -  \beta_h(V^{(i)}_t) h^{(i)}_t] dt. 
	\end{array}\right.
\end{equation}
\item\label{item:20221228c} Write $(\tau^{(j)}_n)_{n\in\mathbbm{N}}$ for the sequence of 
spike times of neuron 
$\mathcal{N}^{(j)}$, and associate a counting process  $N^{(j)}=(N^{(j)}_t)_{t\ge 0}$ to  
$(\tau^{(j)}_n)_{n}$. From  $N^{(j)}$ we define  an output process $U^{(j)}$ for neuron 
$\mathcal{N}^{(j)}$      
\begin{equation}\label{def:output_processes_circuit} 
	dU^{(j)}_t =  -c_1 U^{(j)}_{t-} dt + dN^{(j)}_t , t\ge 0    
\end{equation}
and stress that we use for all  $1\le j\le N$ the same decay parameter $c_1\approx c^*$ 
selected as in modelization step~\ref{subsection:53}. This implies that for all neurons  
$\mathcal{N}^{(j)}$, the 
same benchmarks 
\eqref{def:approxcstar} define an interval $(u_1^*,u_2^*)$ over which values of  $U^{(j)}$ are 
expected to fluctuate in case of regular spiking once time is large enough.  
\item\label{item:20221228d} At this stage, the structure of the processes  $A^{(i)}$ in 
\eqref{circuit_single_HH} 
can 
be specified as follows: 
\begin{eqnarray*}
	A^{(i)}_t := \Phi_{\mathrm{exc}}\left( U^{(i-1)}_{t-} \right) \quad,\quad i\in 
	I_{\mathrm{exc}},  \\
	A^{(i)}_t := \Phi_{\mathrm{inh}}\left( U^{(i-1)}_{t-} \right) \quad,\quad i\in 
	I_{\mathrm{inh}}. 
\end{eqnarray*} 
Here we use  \eqref{def:output_processes_circuit}, 
\eqref{cond_psinod}--\eqref{def_transition_functions},   \eqref{def:inh_exc_subsets}, 
and count  modulo $N$ around the circuit  \eqref{circuit_N_neurons_in_M_blocs}. 
In particular, at time~$t$, neuron $\mathcal{N}^{(1)}$ depends via 
\[
A^{(1)}_t = \Phi_{\mathrm{inh}}\left( U^{(N)}_{t-} \right)
\]
on the output of neuron $\mathcal{N}^{(N)}$ immediately before time $t$.  
\item To initialize the circuit  \eqref{circuit_N_neurons_in_M_blocs}, we  sample starting 
values  
\begin{enumerate}[i)]
\item $X^{(i)}_0$  --for the Ornstein-Uhlenbeck processes in \ref{item:20221228a})-- 
from the invariant law $\mathcal{N}(0, \frac{\sigma^2}{2\tau})$,  independently for all 
neurons; 

\item $(V^{(i)}_0,n^{(i)}_0,m^{(i)}_0,h^{(i)}_0)$ --for the biological variables in 
\ref{item:20221228b})-- 
from the uniform law on $(-12,120){\times}(0,1)$, independently for all neurons; 

\item[\mylabel{item:pointiiipage24}{iii)}] $U^{(i)}_0$ --for the output processes in 
\ref{item:20221228c})--  either: as random 
variables  
\begin{equation}\label{scenario3}
	\mbox{ $U^{(j)}_0$ independent, $1\le j\le N$, and distributed uniformly on $(1,u^*_1)$ 
	} , 
\end{equation}
or: deterministically  
\begin{equation}\label{scenario1}
	U^{(j)}_0 := 0 \quad\mbox{for all}\quad j=1,\ldots,N . 
\end{equation}
\addtocounter{enumii}{1}
\item In a last step, given the values selected in \ref{item:pointiiipage24} and defining by 
convention 
$U^{(i-1)}_{0-} 
:=U^{(i-1)}_0$, we determine starting values  $A^{(i)}_0$ for the input processes in 
\eqref{circuit_single_HH}, depending  on the output  $U^{(i-1)}_{0-}$ of the  
predecessor of neuron  $\mathcal{N}^{(i)}$ and on its position in the circuit, according to 
\ref{item:20221228d}) 
above.  
\end{enumerate}
This finishes the initialization of the circuit  \eqref{circuit_N_neurons_in_M_blocs}.
\end{enumerate}

We now explain why and  in which sense circuits constructed as explained in modelization 
steps~\ref{subsec:51}, \ref{subsection:53}, \ref{subsection:54}, \ref{subsection:55} will exhibit 
auto-generated rhythmic oscillation of spiking activity around the circuit. There is strong 
evidence from simulations, see figures \ref{fig:circuit_UJ0=0} and \ref{fig:circuit_Uj0random} 
as  two examples. We have no rigorous proofs so far. 
\end{modstep}
~
\begin{remark}\label{rem:56}
	We explain the behaviour of the circuit in the following points 
	\ref{item:641})--\ref{item:645}):  
	\begin{enumerate}[i)]
	\item\label{item:641} For a neuron  $\mathcal{N}^{(i)}$ whose predecessor belongs to the 
	same block, i.e.\ 
	for 
	$\mathcal{N}^{(i)}$ with index  $i\in I_{\mathrm{exc}}$: 
	\begin{enumerate}
		\item[\mylabel{item:641alpha}{$\alpha$)}] 
	Regular spiking of the predecessor  
	$\mathcal{N}^{(i-1)}$  over some 
	amount of 
	time drives values of its output process $U^{(i-1)}$ into neighbourhoods of the interval 
	$(u^*_1,u^*_2)$, hence values of  $A^{(i)}=\Phi_{\mathrm{exc}}(U^{(i-1}))$ towards 
	$\vartheta_2$. Quite rapidly, this will force  $\mathcal{N}^{(i)}$  into a regime of regular 
	spiking. 
\item[\mylabel{item:641beta}{$\beta$)}]  Quiet behaviour of the predecessor 
$\mathcal{N}^{(i-1)}$  over some 
	amount of 
	time forces its output  $U^{(i-1)}$ exponentially fast towards $0$, hence values of  
	$A^{(i)}=\Phi_{\mathrm{exc}}(U^{(i-1)})$ towards $\vartheta_1$. As a consequence,  
	neuron  
	$\mathcal{N}^{(i)}$ will soon be silenced and enter the quiet regime.  
	\end{enumerate}
	
	\item For a neuron $\mathcal{N}^{(i)}$ which occupies the first place in its block, i.e.\ for 
	$\mathcal{N}^{(i)}$ with  index  $i\in I_{\mathrm{inh}}$: 
	\begin{enumerate}
	\item[\(\alpha\))] Regular spiking of the predecessor over some amount of time drives its 
	output 
	$U^{(i-1)}$ into neighbourhoods of the interval $(u^*_1,u^*_2)$, hence values of  
	$A^{(i)}=\Phi_{\mathrm{inh}}(U^{(i-1)})$ down to $\vartheta_1$. As a consequence, 
	neuron  
	$\mathcal{N}^{(i)}$ will be silenced and enter the quiet regime.  
	\item[\(\beta\))] 
	Quiet behaviour of the predecessor over some amount of time forces its output  
	$U^{(i-1)}$ towards $0$, hence values of  $A^{(i)}=\Phi_{\mathrm{inh}}(U^{(i-1)})$ towards 
	$\vartheta_2$. As a consequence, neuron  $\mathcal{N}^{(i)}$ will be forced into a 
	regime of regular spiking. 
	\end{enumerate}
	\item\label{item:point3page25} When the blocks have suitable length $L$, the following 
	happens inside every block:
		\begin{enumerate}
		\item[\(\alpha\))]  excitation of successor neurons as in \ref{item:641} \ref{item:641alpha} 
		tends to propagate 
		through the 
	block as a whole (i.e.\ with high probability, the neuron in the last position of the block will 
	get 
	excited before randomness might generate other patterns), and finally all neurons in this 
	block will be regularly spiking; 
		\item[\(\beta\))] silencing of successor neurons as in  \ref{item:641}  \ref{item:641beta}  
		tends to  
		propagate 
		through the 
	block as a whole (with analogous caveat), and finally all neurons in the block will be 
	silenced.  
	\end{enumerate}
	Thus a pattern  distinguishing active blocks from quiet blocks will appear in the circuit.  
	\item Consider neurons $\mathcal{N}^{(i)}$ which occupy the first position $i\in 
	I_{\mathrm{inh}}$ in 
	their block. Then together with $\mathcal{N}^{(i)}$, the block as a whole  --as described in 
	\ref{item:point3page25})-- will 'flip' whenever the following happens : 	
	\begin{enumerate}
		\item[\(\alpha\))]From active to quiet, given that the block so far was regularly spiking:  
	at the time where activity propagating through the preceding block reaches position $i{-}1 
	\in 
	I_{\mathrm{exc}}$ and thus excites neuron $\mathcal{N}^{(i-1)}$, input $A^{(i)}=\Phi_{\rm 
	inh}(U^{(i-1)})$ will break down 
	(as a consequence of growing output $U^{(i-1)}$ since $i\in I_{\mathrm{inh}}$), thus neuron 
	$\mathcal{N}^{(i)}$ will be forced into silence.  
	\item[\(\beta\))]   From quiet to active, given that the block so far was quiet:  
	at the time where silence propagating through the preceding block attains position $i{-}1$ 
	and  silences 
	neuron $\mathcal{N}^{(i-1)}$, input $A^{(i)}=\Phi_{\mathrm{inh}}(U^{(i-1)})$ will increase 
	for 
	position $i$ (as a consequence of decreasing  output $U^{(i-1)}$ since $i\in 
	I_{\mathrm{inh}}$), 
	thus neuron $\mathcal{N}^{(i)}$ will be forced into regular spiking regime. 
	In both cases, by \ref{item:point3page25}), silence of $\mathcal{N}^{(i)}$ or activity of 
	$\mathcal{N}^{(i)}$  will 
	propagate over the corresponding block.  
	\end{enumerate}
	\item\label{item:645} The number $M$ of blocks being odd by assumption in 
	modelization step~\ref{subsection:55}, stable 
	coexistence in 
	equal 
	number of blocks which are permanently quiet alternating with blocks which are 
	permanently 
	active is impossible; in permanence, blocks will be obliged to 'flip' when either some active 
	block is approached from the left by propagating activity, or some silent block is 
	approached 
	from the left by silence. By 'flipping' of suitable blocks at suitable times, the pattern of 
	alternating active and quiet regions around the circuit performs some kind of 
	counter-clockwise rotation which looks very much like  a periodic phenomenon. 
	\end{enumerate}
\end{remark}

Figures \ref{fig:circuit_UJ0=0} and \ref{fig:circuit_Uj0random} below illustrate how oscillating 
activity patterns according to blocks in circuits described in 
modelization steps~\ref{subsec:51}, \ref{subsection:53}, \ref{subsection:54}, 
\ref{subsection:55} appear 
and 
stabilize in a slow rhythmic 
rotation around the circuit. Both figures use the same delay parameter $c_1$  chosen 
according to \eqref{def:approxcstar}, and the same parameter values for $\tau$ and $\sigma$ 
chosen as in modelization step~\ref{subsec:51}. The choice of starting values for the 
collection of output 
processes 
is somewhat different: figure \ref{fig:circuit_UJ0=0} has $U^{(j)}_0=0$ for all $j$ as in 
\eqref{scenario1}, 
and figure \ref{fig:circuit_Uj0random} selects for all $j$ an initial position  uniformly on 
$(1,u^*_1)$ as in \eqref{scenario3}. In both cases, after some  initial phase of randomness 
which prevails in all blocks, spiking activity in one block turns out to be strong enough to 
silence 
its successor block, thus initializing a rotative motion of silent and active regions along the 
circuit.

\begin{remark}\label{rem:57}
	With reference to Ditlevsen and L\"ocherbach \cite{DL-17}, we discuss a deterministic 
	reference model which explains why we expect the self-organized rhythmic behaviour of 
	the system constructed in modelization step~\ref{subsection:55} --illustrated by figures 
	\ref{fig:circuit_UJ0=0} and 
	\ref{fig:circuit_Uj0random}, and explained in remark~\ref{rem:56}-- to be persistent in the 
	long run, certainly from time to time perturbed in a random way but always restoring itself 
	rapidly in the sequel. The reference model is a simplified special case of the deterministic 
	limit model in \cite{DL-17}. 
	
	Think of a deterministic system of dimension $N=ML$ where real-valued variables $t \to 
	x_i(t)$ represent in some way a spiking activity of neuron $i$ as a function of time, with 
	neurons arranged as a circuit of $M$ blocks of $L$ neurons, and where counting modulo 
	$N$ the interaction is of type   
	\[
	\dfrac{d x_i}{dt}(t) = \left\{ \begin{array}{ll}
		-c x_i(t)  -  f(x_{i-1}(t))  &\quad\mbox{if}\quad  i\in \mathbbm{I}_{\mathrm{inh}}   \\
		-c x_i(t) + f(x_{i-1}(t))   &\quad\mbox{if}\quad  i\in \mathbbm{I}_{\mathrm{exc}}
	\end{array} \right. 
	\]
	with $f$ some smoothed version of the truncated identity $x \to (x\vee {-}1)\wedge 1$, and 
	$c\in (0,1)$ some constant. As in \eqref{def:inh_exc_subsets}, indices  
	$\mathbbm{I}_{\rm 
		inh}$ correspond to neurons which occupy  the first position in their block. 
	
	Under the condition that i) $M$ is odd and ii) $c$ is small enough, this system evolves on a 
	finite number of periodic orbits, and at least one periodic orbit is stable. This follows from 
	Theorem~3 in Section~4 of \cite{DL-17}. 
	Random initial conditions in this deterministic model ($(x_1(0),\ldots,x_N(0))$ drawn from a 
	uniform law on $(-1,1)^N$) produce activity patterns very similar to what we see in figures 
	\ref{fig:circuit_UJ0=0} and \ref{fig:circuit_Uj0random}.
\end{remark}

\begin{remark}\label{rem:58}
	We emphasize that modelization step~\ref{subsec:51} requires regular spiking under 
	$\vartheta_2$ and 
quiet behaviour under $\vartheta_1$, both  with probability close to $1$, under the same pair 
$(\tau,\sigma)$ governing the Ornstein-Uhlenbeck noise in all neurons. Choice of  
$(\tau,\sigma)$ 
is of key importance for the feature of self-organized oscillation in systems constructed in 
modelization steps~\ref{subsec:51}, \ref{subsection:53}, \ref{subsection:54}, 
\ref{subsection:55} of 
interacting stochastic Hodgkin-Huxley models. The feature of interest --self-organized slow 
rhythmic oscillation of activity patterns around the circuit-- will be destroyed when the 
volatility 
$\sigma$ becomes too large or the back-driving force $\tau$ too small: then for all neurons in 
the 
circuit, the spiking activity will be more or less irregular or chaotic. In this sense, 
Definitions~\ref{def:31} and \ref{def:41} are of key importance for our construction.  
\end{remark}

\begin{figure}[h]
	\centering
	\includegraphics[width=0.95\linewidth]{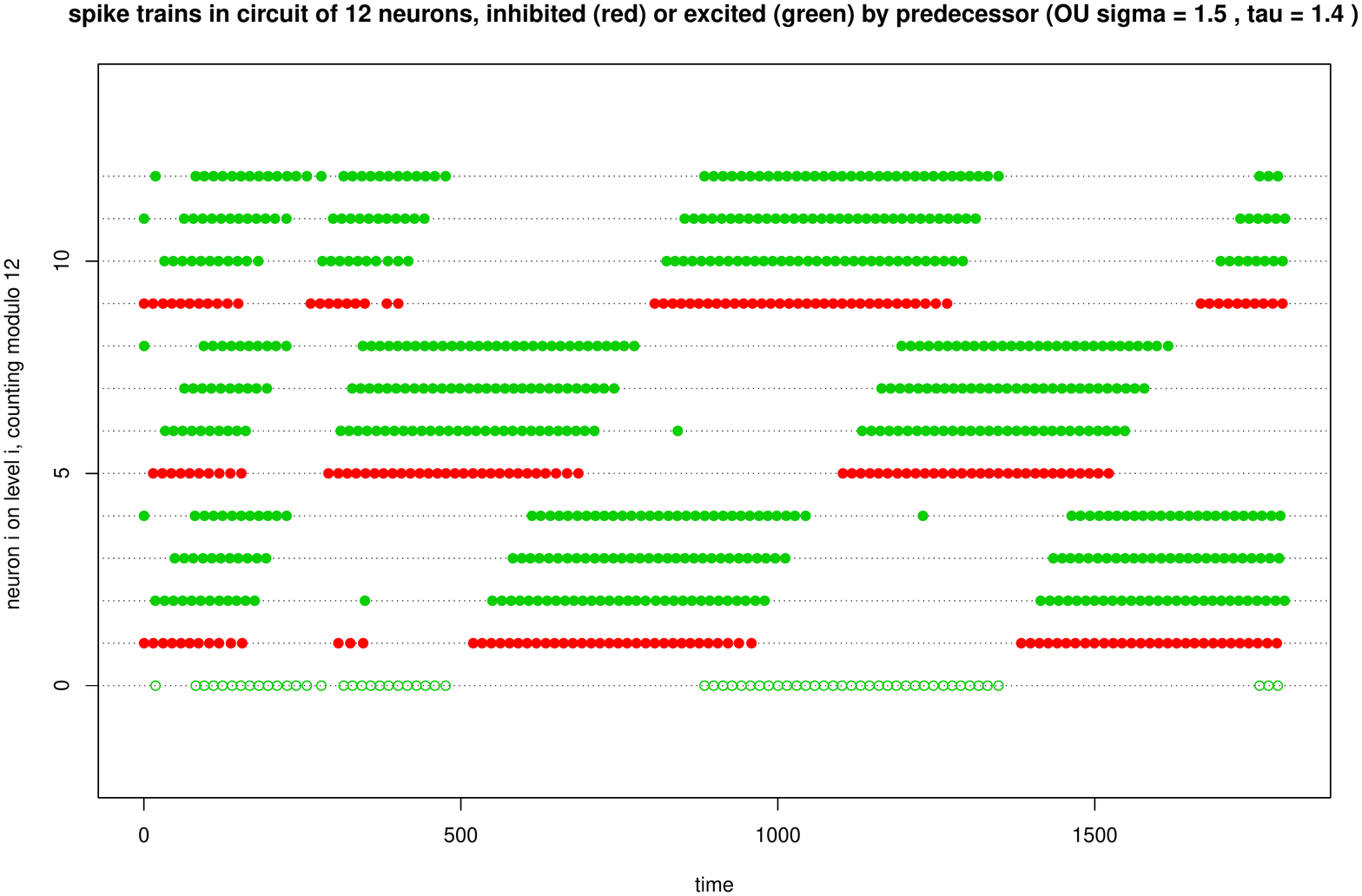} 
	\caption{
		Simulation of a circuit described in modelization 
		steps~\ref{subsec:51}, \ref{subsection:53}, \ref{subsection:54} and \ref{subsection:55} 
		up to 
		time 1800, using Euler 
		schemes of step size 0.001, 
		with $N=12$ neurons in $M=3$ blocks of $L=4$ cells.  
		We show the spike times of all neurons, with neuron $\mathcal{N}^{(i)}$ represented at 
		horizontal level $i$, and $\mathcal{N}^{(12)}=\mathcal{N}^{(0)}$ represented twice to 
		visualize the cyclic structure.  
		We use red dots for positions $i\in \mathbbm{I}_{\mathrm{inh}}$, green dots for $i\in 
		\mathbbm{I}_{\mathrm{exc}}$. 'Noise' has 
		parameter values $\sigma=1.5$ and $\tau=1.4$, as in Example~\ref{ex:52}. The decay 
		parameter 
		$c_1=0.02$ for output processes satisfies \eqref{def:approxcstar}, we have 
		$\Delta^*\approx 14.3$ in \eqref{def_nustar} and $u_1^*\approx 3.01$ in 
		\eqref{nustar+ustars}. 
		Initial conditions for output processes are \eqref{scenario1}: $U^{(j)}_0= 0$ for all $j$. 
		Thus the
		input processes $A^{(i)}$ start at $\vartheta_2=10$ for $i\in 
		\mathbbm{I}_{\mathrm{inh}}$, and 
		at 
		$\vartheta_1=4$ for $i\in \mathbbm{I}_{\mathrm{exc}}$; this gives a slight 'advantage' to 
		positions 
		$i\in \mathbbm{I}_{\mathrm{inh}}$ which tend to spike earlier,  exciting successor 
		neurons in 
		the 
		same block. In the initial phase of this simulation, the block $\{9,10,11,12\}$ is the first 
		which 
		succeeds in silencing its successor block. 
	}
	\label{fig:circuit_UJ0=0}
\end{figure}

\begin{figure}[h]
	\centering
	\includegraphics[width=0.95\linewidth]{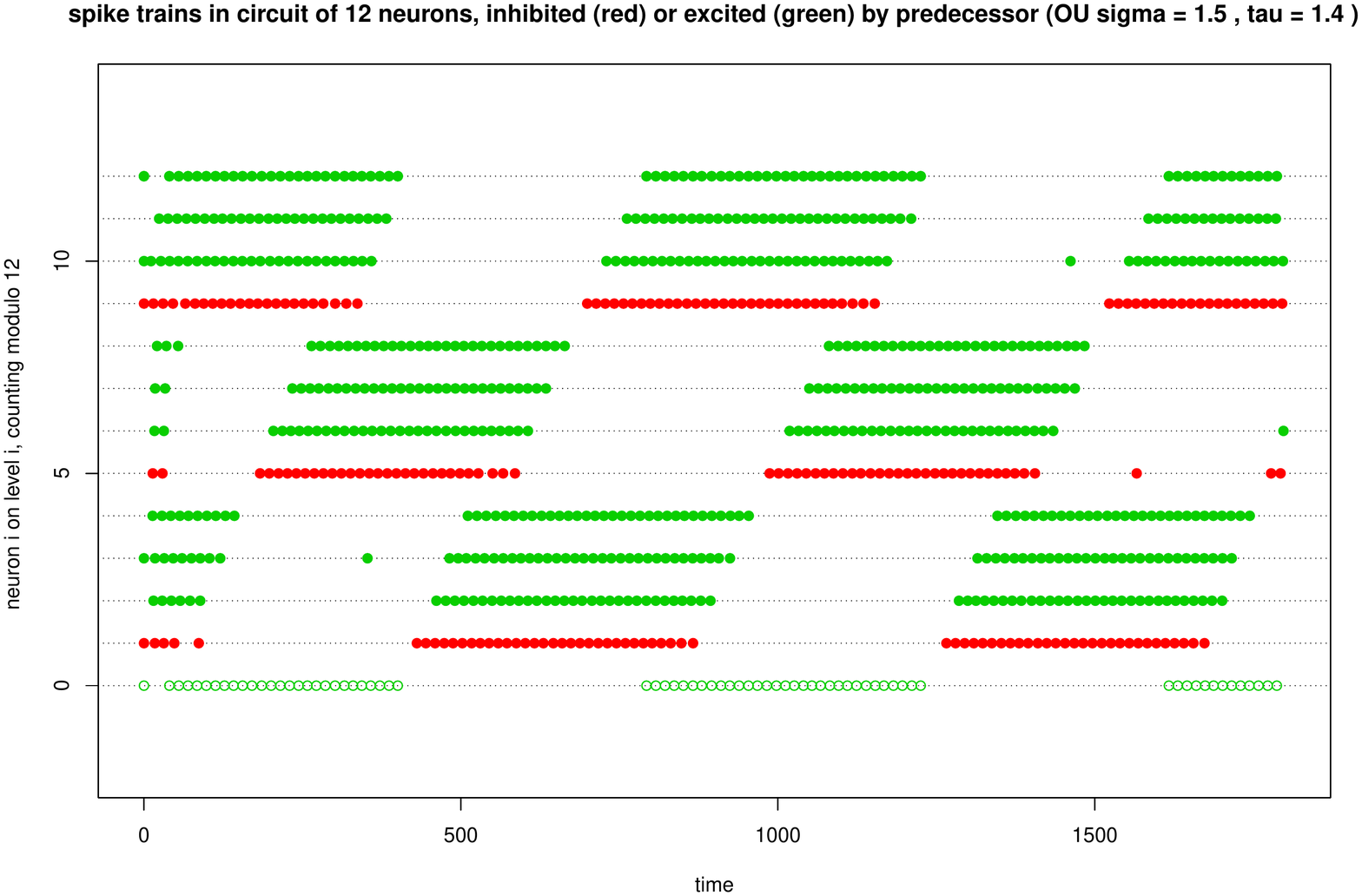} 
	\caption{\small 
		Simulation of a circuit with the same structure and the same parameters as in figure 
		\ref{fig:circuit_UJ0=0}, except that now we use random initial conditions  
		\eqref{scenario3} for the output processes: for all $j$,   $U^{(j)}_0$ is distributed 
		uniformly on $(1,u^*_1)$. Quite rapidly, the circuit organizes itself in patterns which 
		block-wise perform a slow rotation around the circuit. 
	}
	\label{fig:circuit_Uj0random}
\end{figure}

%
\section{Appendix : details for some proofs in Section~\ref{sect:stoch_HH} 
}\label{sect:proof_prop_2.5} 
%

Fix $\vartheta$, $\tau$ and $\sigma$ and suppress corresponding superscripts ($Q_x = 
Q_x^{(\vartheta,\tau,\sigma)}$,  $E_\mu = E_\mu^{(\vartheta,\tau,\sigma)}$, etc.). 

Under the lower bound condition given by Theorem~\ref{theo:21}~\ref{item:21d}), with $T$, 
$\alpha$, 
$\nu$ and $C$ as there, Nummelin splitting in the grid chain 
$(\mathbb{X}_{kT})_{k\in\mathbbm{N}_0}$ works as follows. 
Prepare i.i.d.\ random variables $(V_k)_{k\in\mathbbm{N}_0}$, uniformly distributed on 
$(0,1)$, and independent of the process  $(\mathbb{X}_{kT})_{k\in\mathbbm{N}_0}$. 
Whenever the grid chain enters the 'small set' $C$ at a time $k$ in a state $x\in C$, we split 
the transition away from $x$ according to the value of $V_k$: on $\{V_k<\alpha\}$, we select 
the successor state $y$ for $x$ according to $\nu(dy)$; on $\{V_k\ge \alpha\}$, we select  
$y$ according to the probability measure $\frac{P_T(x,dy)-\alpha\nu(dy)}{1-\alpha}$. Apply 
colors as follows: on $\{V_k<\alpha\}$ we color $(k,x)$ 'red' and  $(k{+}1,y)$ 'green'; on 
$\{V_k\ge \alpha\}$ we color both $(k,x)$ and  $(k{+}1,y)$ 'blue'. All other transitions remain 
uncolored. 
This amounts to an extension of the underlying probability space such that for the 'colored' 
grid chain $(\mathbb{X}_{kT})_{k\in\mathbbm{N}_0}$, the set of 'green' time points defines a 
sequence of renewal times where the grid chain starts afresh from law $\nu$. This is 
Nummelin \cite{Num-78}. 
If we define 
\[
R_0\equiv 0 \quad,\quad 
R_{n+1} := \inf \left\{ k>R_n : \mathbb{X}_{kT} \in C\quad\mbox{and}\quad V_k < \alpha 
\right\} 
, n\in \mathbbm{N}_0  ,  
\]
then $(R_n)_n$ is a sequence of stopping times with respect to the discrete filtration 
\[
\check{\mathbb{F}} = (\check{\mathcal{F}_k})_{k\in\mathbbm{N}_0}
\quad,\quad \check{\mathcal{F}_k} := \sigma\left( \mathbb{X}_{jT} , V_j : 0\le j\le k \right) . 
\]
Harris recurrence implies $R_n\uparrow \infty$ almost surely as $n\to\infty$. 
In discrete time, successive path segments from 'green' to subsequent 'red' times (i.e.\ from 
$R_n+1$ to $R_{n+1}$, $n\in\mathbbm{N}$) decompose the trajectory of 
$(\mathbb{X}^{(\vartheta)}_{kT})_{k\in\mathbbm{N}_0}$ into i.i.d.\ excursions which we call 
life 
cycles, up to some initial segment. Positive Harris recurrence by Theorem~\ref{theo:21} grants 
that the expected length $E_\bullet(R_2-R_1)$ of a life cycle is finite (and independent of the 
starting point of $\mathbb{X}$).

In continuous time we consider the filtration 
\begin{equation}\label{extended_cont_filtration}
	\mathbb{F} := \left( \mathcal{F}_t \right)_{t\ge 0}  \quad,\quad 
	\mathcal{F}_t := \bigcap_{r>t} \mathcal{F}^{\circ}_r \quad,\quad 
	\mathcal{F}^{\circ}_t := \left\{ \mathbb{X}_s , s\le t  ,  V_j , j T \le  t \right\}  
\end{equation}
generated by the pair 
\[
\left( \mathbb{X}^{(\vartheta)}  ,  \sum_j V_j \mathbbm{1}_{\llbracket jT , (j{+}1)T 
\llbracket}  \right) ,    
\]
with $T$ as above. 
Then $( R_n T )_{n\in\mathbbm{N}}$ is a sequence of $\mathbb{F}$-stopping times 
increasing to $\infty$. 
If we think of the continuous-time process in terms of bridges pasted into the grid chain, 
Nummelin splitting in the grid chain with coloring as above shows that path segments 
\begin{equation}\label{def_1_life_cycle}
	\mathbb{X} \mathbbm{1}_{ \llbracket (R_n+1)T , R_{n+1}T \rrbracket } \quad,\quad 
	n\in\mathbbm{N}
\end{equation}
from 'green' to subsequent 'red' times are i.i.d.\  in the continuous-time setting (\cite{HL-03}); 
note that here we do leave out a short piece of trajectory from 'red' to 'green', of length $T$, 
for all $n$. For all $n\in\mathbbm{N}$, the $\mathbb{F}$-stopping times $(R_n{+}1)T$ are 
renewal times where the process starts anew from initial law $\nu$; the future following time 
$(R_n{+}1)T$ is independent from the past $\mathcal{F}_{R_nT}$ up to time $R_n T$.

When we prefer to consider path segments 'from green to green'  
\begin{equation}\label{def_2_life_cycle}
	\mathbb{X} \mathbbm{1}_{ \llbracket (R_j+1)T , (R_{j+1}+1)T \llbracket } \quad,\quad 
	j\in\mathbbm{N}, 
\end{equation}
then  these also are identical in law, but independence holds only two-by-two (cf.
L\"ocher\-bach and Loukianova \cite{LL-08}): those with  $j$ even are i.i.d., and  
--separately-- 
those with $j$ odd.  The same reasoning shows that for every $K\in\mathbbm{N}$ which we 
keep fixed, path segments 'from green to green over $K-1$ renewal intervals'  
\begin{equation}\label{def_3a_life_cycle}
	\mathbb{X} \mathbbm{1}_{ \llbracket (R_{jK+r}+1)T , (R_{(j+1)K+r-1}+1)T \llbracket } 
	\quad,\quad 0< 
	r\le K,
	j\in\mathbbm{N}_0   
\end{equation} 
are identical in law, but independence holds $K$-by-$K$ only: fixing $r$, path 
segments        
\begin{equation}\label{def_3b_life_cycle}
	\mathbb{X} \mathbbm{1}_{ \llbracket (R_{jK+r}+1)T , (R_{(j+1)K+r-1}+1)T \llbracket } 
	\quad,\quad 
	j\in\mathbbm{N}_0   
\end{equation}
are independent when $0< r\le K$ is fixed.  
We shall speak of \eqref{def_2_life_cycle} or of \eqref{def_3a_life_cycle} as life cycles. 
\begin{lemma}\label{lem:61}
	Fix a natural number $L>1$. For every point  $v=(v_1,v_2, \ldots, v_L) \in 
[0,\infty)^L$,  with notation $[0,v]:=\mathop{\mathsf{X}}\limits_{i=1}^L[0,v_i]$, and every 
function  
$h:[0,\infty)^L\to\mathbb{R}$ which is measurable and bounded, writing 
$\mathbbm{1}_{[0,v]} h$ for the product $\mathbbm{1}_{[0,v]}\cdot h$, we have almost 
sure convergence of 
\begin{equation}\label{def:G_hat_n}
	\widehat{G}_m( v , h ) := \frac{1}{m} \sum_{n=1}^m \left( \mathbbm{1}_{[0,v]} h \right) 
	\left( \tau_{n+1} - \tau_n  ,  \ldots   ,   \tau_{n+L} - \tau_{n+L-1}  \right)   
\end{equation}
to a deterministic limit as $m\to\infty$.
\end{lemma}
\begin{proof} 
	We  modify the proof of Theorem~2.9 in \cite{HLT-17}, section 5.  
	
	\begin{enumerate}[1)] 
	\item\label{item:20221228_1} We count spikes in the life cycles  \eqref{def_2_life_cycle}: 
	\begin{equation}\label{def:Z_j_tilde}
		\widetilde{Z}_j := \sum_{n=1}^\infty \mathbbm{1}_{ \{ (R_j+1) T \le \tau_n < 
		(R_{j+1}+1) T 
		\} } 
		\quad,\quad j\in\mathbbm{N}. 
	\end{equation}
	Since interspike times are $> \delta_0$ by construction, $\widetilde{Z}_j$ is upper bounded 
	by 
	$\frac{1}{\delta_0} (R_{j+1}-R_j) T$ which has finite expectation. Thus the 
	$(\widetilde{Z}_j)_j$, 
	identical in law, belong to $L^1(Q_\mu)$, and are  independent two-by-two using 
	\eqref{def_2_life_cycle}. As a consequence, 
	\begin{eqnarray*}
		\lim_{m\to\infty} \frac{1}{2m} \sum_{j=1}^{2m} \widetilde{Z}_j =
		\frac{1}{2}  \lim \frac{1}{m} \sum_{j=1}^m \widetilde{Z}_{2j} + \frac{1}{2}  \lim 
		\frac{1}{m} 
		\sum_{j=0}^{m-1} \widetilde{Z}_{2j+1}   
		=  E_\bullet\left( \widetilde{Z}_1  \right) 
	\end{eqnarray*}
	exists almost surely, and thus also $\lim \frac{1}{m} \sum_{j=1}^m \widetilde{Z}_j$. The limit 
	does 
	not 
	depend on the starting point. Now asymptotically as $m\to\infty$
	\begin{eqnarray*}
		\frac{1}{m} \sum_{j=1}^m  \widetilde{Z}_m   &=&   \frac{1}{m} \cdot \mbox{number of 
		spikes in} \;
		\llbracket 
		(R_1+1)T , (R_{m+1}+1)T \llbracket   \\
		&=&  \frac{1}{m} \cdot\, \mbox{number of spikes in}\; \rrbracket (R_1+1)T , 
		(R_{m+1}+1)T 
		\rrbracket  
		+ o_{Q_x}(1) \\
		&=&  \frac{1}{m} \cdot\, \mbox{number of spikes in}\; \rrbracket 0, (R_{m+1}+1)T 
		\rrbracket  
		+
		o_{Q_x}(1) \\
		&=&  \frac{1}{m} N_{ (R_{m+1}+1)T } + o_{Q_x}(1)
	\end{eqnarray*}
	whence with the help from Proposition~\ref{prop:24} 
	\begin{equation}\label{lim_Z_j_tilde}
		E_\mu\left( \widetilde{Z}_1 \right) =
		\lim_{m\to\infty} \frac{1}{m} \sum_{j=1}^m \widetilde{Z}_j =  T E_\mu(N_1)
		E_\mu(R_2-R_1). 
	\end{equation}
	\item\label{item:20221228_2} Fix $v=(v_1,v_2, \ldots, v_L) \in [0,\infty)^L$ and  
	$h:[0,\infty)^L\to\mathbb{R}$ 
	measurable and bounded, and  consider 
	\begin{equation}\label{def:Zj_v_h}
		Z_j( v , h ) := \sum_{n=1}^\infty \mathbbm{1}_{\{ (R_j+1) T \le \tau_n < (R_{j+1}+1) 
		T 
			\} }  \left( \mathbbm{1}_{[0,v]} h \right) \left( \tau_{n+1}-\tau_n , \ldots , 
		\tau_{n+L}-\tau_{n+L-1}  \right)  
	\end{equation} 
	for $j\in\mathbbm{N}$. Then the $(Z_j(v,h)_j$ are identical in law, and belong to 
	$L^1(Q_\mu)$ by comparison with \eqref{def:Z_j_tilde}.  
	We shall show in two steps --\ref{item:pointipage30}) and \ref{item:pointiipage30}) below-- 
	that the variables $(Z_j(v,h)_j$ in 
	\eqref{def:Zj_v_h} are independent $K$-by-$K$ provided we choose $K\in\mathbbm{N}$ 
	large enough, i.e. for 
	\begin{equation}\label{choice_K}
		K \ge  2 +  \frac{1}{T} \sum_{i=1}^L v_i    . 
	\end{equation}
	\begin{enumerate}[i)]
		\item\label{item:pointipage30} In \eqref{def:Zj_v_h} we have to consider events 
	\[
	B(j,n) := \left\{(R_j+1) T \le \tau_n < (R_{j+1}+1) T \right\} 
	\bigcap\bigcap_{i=1}^L
	\left\{ \tau_{n+i}-\tau_{n+i-1} \in [0,v_i] 
	\right\} 
	\]
	on which 
	\[
	\tau_{n+L}  =   \tau_n + \sum_{i=1}^L (\tau_{n+i}-\tau_{n+i-1})  < (R_{j+1}+1) T + 
	\sum_{i=1}^L v_i   ; 
	\]
	thus with the choice \eqref{choice_K}, the following holds on $B(j,n)$: 
	\[
	\tau_{n+L}  < (R_{j+1}+1) T + (K-2) T = (R_{j+1}+K-1) T <  R_{j+K}T. 
	\]
	\item\label{item:pointiipage30} Now 
	\[
	\left\{ (R_j+1) T \le \tau_n < (R_{j+1}+1) T \right\}\in \mathcal{F}_{\tau_n}
	\]
	and $(\tau_j)_j$ is an increasing sequence of $\mathbb{F}$-stopping times: thus 
	the event $B(j,n)$  belongs to the $\sigma$-field $\mathcal{F}_{R_{j+K}T}$ of events up 
	to 
	time $R_{j+K}T$. Since the last $\sigma$-field does not depend on $n$, the definition 
	\eqref{def:Zj_v_h} grants that for every $j\in\mathbbm{N}$, 
	\begin{equation}\label{indep_a_Zj}
		\mbox{ 
			the variable $Z_j( v , h )$ is measurable with respect to $\mathcal{F}_{R_{j+K}T}$  }
	\end{equation}
	whereas the construction of the renewal times $(R_j+1) T$ for the continuous-time process 
	$\mathbb{X}$ implies 
	\begin{equation}\label{indep_b_Zj}
		\mbox{ 
			the variable $Z_j( v , h )$ is independent of  $\mathcal{F}_{R_j T}$  }
	\end{equation}
	for all $j$. As a consequence, the family $(Z_j( v , h ))_j$ is independent $K$-by-$K$ as 
	asserted since 
	we have for every  $0< r\le K$  independence in restriction to  the subfamily   
	$\left\{ Z_{\ell K+r}( v , h )  ,  \ell\in\mathbbm{N}_0 \right\}.$
	\end{enumerate}
	\item\label{item:20221228_3}  The  $(Z_j( v , h ))_j$ being identical in law and 
	independent  $K$-by-$K$, a 
	deterministic limit 
	\begin{align*}
	\lim_{m\to\infty} \frac{1}{m} \sum_{j=1}^m Z_j(v,h) &= 
	\lim_{m\to\infty} \frac{1}{mK} \sum_{j=1}^{mK} Z_j(v,h)\\
	& = 
	\frac{1}{K} \sum_{r=1}^K \left( \lim_{m\to\infty} \frac{1}{m} \sum_{\ell=0}^{m-1} Z_{\ell 
	K+r}(v,h) 
	\right) = E_\mu(Z_1(v,h)) 
	\end{align*}
	exists almost surely, by the classical strong law of large numbers.  This is the essential step 
	in the proof of the lemma. 
	\item\label{item:20221228_4} Asymptotically as $m\to\infty$ we can write  
	\[
	\frac{1}{m} \sum_{j=1}^m Z_j(v,h)
	\]
	in the following form:  
	\begin{eqnarray*}
		&&\frac{1}{m} \sum_{j=1}^m \sum_{n=1}^\infty \mathbbm{1}_{\{ (R_j+1)T \le \tau_n < 
			(R_{j+1}+1)T \}} \left( \mathbbm{1}_{[0,v]} h\right) \left(\tau_{n+1}-\tau_n 
			,\ldots,  
		\tau_{n+L}-\tau_{n+L-1} \right) \\
		&&=\quad  \frac{1}{m}   \sum _{n=1}^\infty \mathbbm{1}_{\{ (R_1+1)T \le \tau_n < 
			(R_{m+1}+1)T \}} \left( \mathbbm{1}_{[0,v]} h\right) \left(\tau_{n+1}-\tau_n 
			,\ldots,  
		\tau_{n+L}-\tau_{n+L-1} \right) \\
		&&=\quad  \frac{1}{m}   \sum _{n= N_{(R_1+1)T} + 1 }^{ N_{(R_{m+1}+1)T} }  
		\left( \mathbbm{1}_{[0,v]} h\right) 
		\left(\tau_{n+1}-\tau_n ,\ldots,  \tau_{n+L}-\tau_{n+L-1} \right) + 
		o_{Q_x}(1)\\
		&&=\quad  \frac{1}{m}   \sum _{n=1}^{ N_{(R_{m+1}+1)T} }  \left( \mathbbm{1}_{[0,v]} 
		h\right) 
		\left(\tau_{n+1}-\tau_n ,\ldots,  \tau_{n+L}-\tau_{n+L-1} \right) + 
		o_{Q_x}(1)\\
		&&=\quad  \frac{ N_{(R_{m+1}+1)T }}{m} \widehat{G}_{ N_{(R_{m+1}+1)T} } (v,h) 
		+ 
		o_{Q_x}(1) \\
		&&=\quad  T E_\mu\left( N_1 \right) E_\mu\left( R_2 - R_1 \right) \widehat{G}_{ 
		N_{(R_{m+1}+1)T} } (v,h) + o_{Q_x}(1)   
	\end{eqnarray*}
	with the help from Proposition~\ref{prop:24} and since $h:[0,\infty)^L\to\mathbb{R}$ is 
	bounded. 
	From the last line and step \ref{item:20221228_3}) it follows that 
	\begin{equation}\label{limes_a_Ghutm}
		\lim_{m\to\infty} \widehat{G}_{ N_{(R_{m+1}+1)T} } (v,h) 
	\end{equation}
	exists almost surely and equals 
	\begin{equation}\label{limes_b_Ghutm}
		E_\mu\left( Z_1(v,h) \right) / \left( T E_\mu\left( N_1 \right) E_\mu\left( R_2 - R_1 
		\right) \right) 
		\quad=\quad E_\mu\left( Z_1(v,h) \right) / E_\mu\left( \widetilde{Z}_1 \right)
	\end{equation}
	with reference to step \ref{item:20221228_1}). 
	Decomposing $h$ into positive and negative part shows that it is sufficient to consider 
	$h\ge 0$: but then, existence of the limit in \eqref{limes_a_Ghutm} is equivalent to 
	existence of the limit 
	\begin{equation}\label{limes_c_Ghutm}
		\lim_{m\to\infty} \widehat{G}_m (v,h) 
	\end{equation}
	almost surely, and the proof of the lemma is finished.
	\end{enumerate}
\end{proof}

\begin{remark}\label{rem:62}
	The almost sure limit in Lemma~\ref{lem:61} and in 
	\eqref{limes_a_Ghutm}--\eqref{limes_c_Ghutm} of 
	its proof
	\[
	E_\mu\left( Z_1(v,h) \right) / E_\mu\left( \widetilde{Z}_1 \right) = 
	\lim_{m\to\infty} \widehat{G}_m(v,h) 
	\]
	admits an interpretation: it equals the relative number of spikes in the long run for which 
	subsequent $L$ interspike times realize a particular pattern, expressed by the function 
	$\mathbbm{1}_{[0,v]}h$, 
	\begin{equation}\label{limes_d_Ghutm}
		\lim_{t\to\infty} \frac{1}{N_t} \sum_{n=1}^{N_t} \left( \mathbbm{1}_{[0,v]} h \right) \left( 
		\tau_{n+1} - \tau_n  ,  \ldots   ,   \tau_{n+L} - \tau_{n+L-1}  \right)      
	\end{equation}
	(considering $h\ge 0$ first,  \eqref{limes_d_Ghutm} is a consequence of 
	\eqref{limes_c_Ghutm} exactly as  \eqref{limes_c_Ghutm} was a consequence of  
	\eqref{limes_a_Ghutm} in step \ref{item:20221228_4}) of the proof of Lemma~\ref{lem:61}), 
	and it equals  --in 
	terms of life 
	cycles, cf.\ \eqref{def:Zj_v_h} and \eqref{def_2_life_cycle}-- the ratio 
	\begin{equation}\label{limes_e_Ghutm}
		\frac{ E_\mu\left( \sum_{n=1}^\infty \mathbbm{1}_{ \{ (R_1+1) T \le \tau_n < 
		(R_2+1) T 
				\} }   
			\left( \mathbbm{1}_{[0,v]} h \right) \left( \tau_{n+1}-\tau_n , \ldots , 
			\tau_{n+L}-\tau_{n+L-1}  \right) \right) }
		{ E_\mu\left( 
			\sum_{n=1}^\infty \mathbbm{1}_{ \{ (R_1+1) T \le \tau_n < (R_2+1) T \} }  
			\right) }
	\end{equation}
	between the expected number of spikes in a life cycle for which subsequent  $L$ interspike 
	times realize this particular pattern, divided by the expected number of spikes in the life 
	cycle.
\end{remark}

\begin{proof}[Proof for Theorem~\ref{theo:25}] The special case $h\equiv 1$ in 
Lemma~\ref{lem:61}
	establishes pointwise convergence on $[0,\infty)^L$  of empirical distribution functions 
	$\widehat{G}_m(\cdot):= \widehat{G}_m(\cdot,h\equiv 1)$ associated to the first $m$ 
	observed $L$-tuples of 
	successive interspike times 
	\[
	\left( \tau_{n+1}-\tau_n , \ldots , \tau_{n+L}-\tau_{n+L-1}  \right) \quad,\quad 
	n\in\mathbbm{N} 
	\]
	in \eqref{Ltupels_ISI} to a limit $G_\mu(\cdot)$. The proof of Lemma~\ref{lem:61}, or 
	Remark~\ref{rem:62}, identifies the limit as  
	\[
	G_\mu(v) := E_\mu\left( Z_1(v,h\equiv 1) \right) / E_\mu\left( \widetilde{Z}_1 \right) 
	\quad,\quad 
	v\in [0,\infty)^L   
	\]
	where $Z_1(v,h\equiv 1)$ and $\widetilde{Z}_1$ are given by \eqref{def:Zj_v_h} and 
	\eqref{def:Z_j_tilde}. By \eqref{limes_e_Ghutm}, with $h\equiv 1$ and $v=(v_1,\ldots,v_L)$, 
	\[
		G_\mu(v) =  
		\frac{ E_\mu\left(\sum_{n=1}^\infty \mathbbm{1}_{ \{ (R_1+1) T \le \tau_n < 
		(R_2+1) T 
				\} }   
			\mathbbm{1}_{[0,v_1]} ( \tau_{n+1}-\tau_n ) \cdots \mathbbm{1}_{[0,v_L]} ( 
			\tau_{n+L}-\tau_{n+L-1} ) \right)  }
		{ E_\mu\left(  \sum_{n=1}^\infty \mathbbm{1}_{ \{ (R_1+1) T \le \tau_n < (R_2+1) T 
		\} } 
			\right) }   . 
	\]
	Interspike times are $>\delta_0$ by construction, are finite, and a life cycle contains a finite 
	number of spikes: so the last representation shows that $G_\mu(\cdot)$ is the distribution 
	function of a probability measure on $[0,\infty)^L$, clearly concentrated on $(0,\infty)^L$. 
	Pointwise convergence $\widehat{G}_m \to G_\mu $ on $[0,\infty)^L$ being established, 
	uniformity on $[0,\infty)^L$ follows as in classical proofs of the Glivenko-Cantelli Theorem 
	on $\mathbb{R}^L$. 
	Theorem~\ref{theo:25} is proved. 
\end{proof}

\begin{remark}\label{rem:64}
	\begin{enumerate}[a)]
	\item \label{item:rem64a} The product $\mathbbm{1}_{[0,v]}h$ in Lemma~\ref{lem:61} 
	allows to study the relative 
	frequencies 
	in 
	the long run of particular patterns in groups of $L$ successive interspike times. 
	As an example, consider points $v$ in $[0,\infty)^L$ such that 
	$v_1=v_2=\ldots=v_L=:\bar{v}$ 
	is sufficiently large, and let $h$ denote the indicator of events in $L$-point data sets such 
	that 
	'the distance between upper and lower 10\% quantiles does not exceed  5\% of the median'. 
	Then $G_\mu(v,h)$ gives the proportion of spike times $\tau_n$ in the long run which are to 
	be followed by interspike times $( \tau_{n+1}-\tau_n , \ldots , \tau_{n+L}-\tau_{n+L-1})$ with 
	the following two properties: i) the interspike times do not exceed $\bar{v}$;  ii) the 
	interspike times cluster in the above sense in small neighbourhoods of their median. Under 
	this (or similar) definition of $h$, spikes  will look close-to-equally spaced over large 
	periods 
	of time whenever $G_\mu(v,h)$ is close to one. 
	$\widehat{G}_m(v,h)$ gives the relative frequency of the pattern encoded in $\left( 
	\mathbbm{1}_{[0,v]} h\right)$ in  observed spike trains of length $m$. While we have 
	probably no chance to calculate  $G_\mu(\cdot,h)$ in the sense of an explicit and 
	closed-form 
	expression, we may replace it with $\widehat{G}_m(\cdot,h)$  when $m$ is large. Thus 
	Lemma~\ref{lem:61}
	provides us --asymptotically as $m\to\infty$-- with tools in view of statistical inference. 
\item \label{item:rem64b}  We emphasize that the renewal techniques in the proof of 
Lemma~\ref{lem:61} build on 
	presence 
	of 
	an indicator $\mathbbm{1}_{[0,v]}$ in the product $\mathbbm{1}_{[0,v]} h$: this indicator 
	can 
	not be omitted. 
	\end{enumerate}
\end{remark}

\begin{proof}[Proof for Proposition~\ref{prop:26}:]  Fix $\varepsilon>0$. Choose $L$ large 
enough 
	for $\sum_{\ell>L} e^{ - c_1   \delta_0 \ell }  < \varepsilon$. 
	For every $n\in\mathbbm{N}$, the pair $\left( U_{ \tau_{n+L} } , U_{ (\tau_{n+L+1})^- }  
	\right)$ to be considered in \eqref{Uafterspike_Ubeforespike_neu} 
	\begin{eqnarray*}
		U_{ \tau_{n+L} } &=& \sum_{j=1}^{n+L}  e^{ - c_1 (\tau_{n+L}-\tau_j) } 
		= \sum_{\ell = 1-n }^{L}  e^{ - c_1 (\tau_{n+L}-\tau_{n+\ell}) }  \\
		U_{ (\tau_{n+L+1})^- } &=&   U_{ \tau_{n+L} }  e^{ - c_1 ( \tau_{n+L+1} - \tau_{n+L} ) }  
		= \sum_{\ell = 1-n }^{L}  e^{ - c_1 (\tau_{n+L+1}-\tau_{n+\ell}) }  
	\end{eqnarray*}
	will be approximated with the help of truncated sums by  
	\begin{eqnarray*}
		V_n   &:=& \sum_{j=n}^{n+L}  e^{ - c_1 (\tau_{n+L}-\tau_j) } 
		= \sum_{\ell = 0 }^{L}  e^{ - c_1 (\tau_{n+L}-\tau_{n+\ell}) }  \\
		V_{n+1}^-   &:=&   V_n  e^{ - c_1 ( \tau_{n+L+1} - \tau_{n+L} ) }  
		= \sum_{\ell = 0 }^{L}  e^{ - c_1 (\tau_{n+L+1}-\tau_{n+\ell}) }  . 
	\end{eqnarray*}
	This pair $\left( V_n , V_{n+1}^- \right)$ appears as approximation  
	\eqref{approx_afterspike_approx_beforespike} in Proposition~\ref{prop:26}. 
	
	\begin{enumerate}[1)]
		\item  Since interspike times are $>\delta_0$ by construction in \eqref{def_spiketimes}, 
		we 
	have geometric bounds 
	\[
	0 < U_{ \tau_{n+L} } - V_n  < \sum_{\ell>L} e^{ - c_1 \ell \delta_0 } 
	<\varepsilon 
	\quad,\quad 
	0 < U_{ (\tau_{n+L+1})^- } - V_{n+1}^-  < \sum_{\ell>L} e^{ - c_1 (\ell+1) 
	\delta_0 }  <\varepsilon 
	\]
	uniformly in $n$. This is the first assertion in Proposition~\ref{prop:26}. 
	\item Consider points $x=(x_1,\ldots,x_{L+1})$ in $[\delta_0,\infty)^{L+1}$,  with $\delta_0$ 
	from \eqref{def_spiketimes}. With $c_1$ from \eqref{def_output_process} or 
	\eqref{output_properties}, we define   continuous functions  
	\begin{eqnarray*}
		f_1 :\quad  x   &\longrightarrow&  e^{ - c_1 ( x_L+\ldots+x_1 ) } +  e^{ - c_1 
		(x_L+\ldots+x_2) } + \ldots +  e^{ - c_1 x_L } + 1   \\
		f_2 :\quad  x   &\longrightarrow&  e^{ - c_1 ( x_{L+1}+\ldots+x_1 ) } +  e^{ - c_1 
		(x_{L+1}+\ldots+x_2) } + \ldots +  e^{ - c_1 (x_{L+1}+x_L) } + e^{ - c_1 x_{L+1}  
		}    
	\end{eqnarray*}
	which on $[\delta_0,\infty)^{L+1}$ are bounded by $\sum_{\ell\ge 0} e^{-c_1 \ell\delta_0}$ 
	(the last bound does not depend on $L$). Interspike times being  $>\delta_0$ by 
	construction, we have 
	\begin{eqnarray*}
		V_n   &=& f_1 \left( \tau_{n+1}-\tau_n , \ldots , \tau_{n+L}-\tau_{n+L-1} , 
		\tau_{n+L+1}-\tau_{n+L}  \right) \\
		V_{n+1}^-   &=&  f_2 \left( \tau_{n+1}-\tau_n , \ldots , \tau_{n+L}-\tau_{n+L-1} , 
		\tau_{n+L+1}-\tau_{n+L}  \right).   
	\end{eqnarray*} 
	We can extend $f_1$, $f_2$ to continuous and bounded functions $[0,\infty)^{L+1} \to 
	\mathbb{R}$.
	\item 
	On $[0,\infty)^{L+1}$, we write indistinctly  $\widehat{G}_m$ for the empirical 
	distribution 
	functions in  Theorem~\ref{theo:25} (with $L$ to be replaced by $L+1$) and for the 
	associated 
	empirical measures 
	\[
	\frac{1}{m} \sum_{n=1}^m  \epsilon_{ \left( \tau_{n+1}-\tau_n , \ldots , 
		\tau_{n+L}-\tau_{n+L-1} , \tau_{n+L+1}-\tau_{n+L}  \right) }  \quad,\quad  m\to\infty   . 
	\]
	By Theorem~\ref{theo:25}, empirical measures  $\widehat{G}_m$  converge weakly in 
	$[0,\infty)^{L+1}$ to $G_\mu$ as $m\to\infty$
	(again we write $G_\mu$ both for the limiting probability measure on $[0,\infty)^{L+1}$ and 
	for its distribution function). Introducing the continuous function  
	\[
	F := \left( f_1 , f_2 \right) : [0,\infty)^{L+1} \to [0,\infty)^2 
	\]
	the continuous mapping theorem shows that the empirical measures 
	\[
	\widehat{H}_m = \frac{1}{m} \sum_{n=1}^m  \epsilon_{ \left( V_n , V_{n+1}^- \right) } 
	= 
	\frac{1}{m} \sum_{n=1}^m  \epsilon_{ F \left( \tau_{n+1}-\tau_n , \ldots , 
		\tau_{n+L}-\tau_{n+L-1} , \tau_{n+L+1}-\tau_{n+L}  \right)   } 
	\]
	on $[0,\infty)^2$,  images 
	of $\widehat{G}_m$ under $F$, converge weakly in  $[0,\infty)^2$ to the probability 
	measure 
	$H_\mu$
	\[
	H_\mu(A) := G_\mu\left( F^{-1}(A) \right) \quad,\quad A\in\mathcal{B}([0,\infty)^2), 
	\]
	the image of $G_\mu$ under $F$. Weak convergence in $[0,\infty)^2$ can be reformulated 
	in 
	terms of distribution functions as asserted in Proposition~\ref{prop:26}.
	\end{enumerate}
\end{proof}

\begin{remark}\label{rem:66}
	It is clear that --introducing some more indices-- the last proof can be 
extended to deal with $J$-tuples \eqref{Uafterspike_Ubeforespike_extended} 
\[
\left(
\left(  U_{\tau_{n+j+L}} ,  U_{(\tau_{n+j+L+1})^-}\right)   _{ 0 \le j \le J } 
\right) \quad,\quad    n\in\mathbbm{N} 
\]
as mentioned at the end of Section~\ref{output}. The problem that we need huge values of 
$L$ 
in order to obtain  small values of $\varepsilon$ remains the same. The limit law in such an 
extension of Proposition~\ref{prop:26} has the interpretation of governing patterns in the 
output process which may be observed in the long run.
\end{remark}

\section{Appendix: a discussion of the benchmarks \eqref{benchmarks_T1} in 
	Section~\ref{sect:regular_spiking} }\label{conjectures}

As a complement to Section~\ref{sect:regular_spiking}, we discuss --with notations of 
Section~\ref{sect:regular_spiking}-- the role of benchmarks \eqref{benchmarks_T1} in 
connection with asymptotic properties of the sequence of events $R(T_1)$ as $T_1\to\infty$. 

Our discussion is based on two conjectures concerning  the limit distribution  
$H^{(\vartheta,\tau,\sigma)}$ from Proposition~\ref{prop:22}~\ref{item:22b})   for the 
empirical distribution 
functions $\widehat{H}_n$ of the first $n$ interspike times as $n\to\infty$.

\begin{conjecture}\label{conjecture:A}
For all parameter values $(\vartheta,\tau,\sigma)$, 
$H^{(\vartheta,\tau,\sigma)}$ is continuous and strictly monotone on its interval of support  
$(\alpha^{(\vartheta,\tau,\sigma)},\beta^{(\vartheta,\tau,\sigma)})$ in $(0,\infty)$. 
\end{conjecture}

In view of the second conjecture, write $\mathtt{q}^{(\vartheta,\tau,\sigma)}(\alpha) := \inf\{  v 
> 0 
: H^{(\vartheta,\tau,\sigma)}(v) \ge  \alpha \}$ for quantiles of  $H^{(\vartheta,\tau,\sigma)}$, 
$\Delta^{(\vartheta,\tau,\sigma)}$ 
for the median of  $H^{(\vartheta,\tau,\sigma)}$, 
$\mathtt{d}^{(\vartheta,\tau,\sigma)}(\alpha)$ for the difference between upper and lower 
$\alpha$-quantiles, and 
\[
\mathtt{r}^{(\vartheta,\tau,\sigma)}(\alpha) :=  
\frac{\mathtt{d}^{(\vartheta,\tau,\sigma)}(\alpha)}{\Delta^{(\vartheta,\tau,\sigma)}} 
\]
for the ratio 'difference between upper and lower $\alpha$-quantiles divided by the median' in 
$H^{(\vartheta,\tau,\sigma)}$.

\begin{conjecture}\label{conjecture:B}
	There are parameter triplets $(\vartheta,\tau,\sigma)$ such that 
$H^{(\vartheta,\tau,\sigma)}$ satisfies 
\[
\mathtt{r}^{(\vartheta,\tau,\sigma)}(0.05)  < 0.3 \quad,\quad 
\mathtt{r}^{(\vartheta,\tau,\sigma)}(0.1)  < 0.2 \quad,\quad 
\mathtt{r}^{(\vartheta,\tau,\sigma)}(0.25)  < 0.1 
\]
together with 
\[
\left| E_\mu^{(\vartheta,\tau,\sigma)}\left( N_1 \right) \Delta^{(\vartheta,\tau,\sigma)}  -  1  
\right| < 0.05  
\]
where we refer to the almost sure limit $\lim\limits_{t\to\infty}\frac{N_t}{t}$ under 
$(\vartheta,\tau,\sigma)$ in virtue of Proposition~\ref{prop:24}.
\end{conjecture}

Even if we have no proof so far (the proofs for Proposition~\ref{prop:22}~~\ref{item:22b})  
--or for 
Theorem~\ref{theo:25}  in Section~\ref{sect:proof_prop_2.5}-- yield existence of almost sure 
limits, and nothing 
more) we do not doubt that  Conjecture~\ref{conjecture:A}  holds true for all parameter 
triplets  
$(\vartheta,\tau,\sigma)$, and that Conjecture~\ref{conjecture:B} holds true whenever 
the signal 
$\vartheta$ is 
large and --depending on the value of $\sigma$-- the back-driving force $\tau$ is large 
enough or 
--depending on the value of $\tau$-- the volatility $\sigma$ is small enough. As an example, 
both 
Conjectures~\ref{conjecture:A} and \ref{conjecture:B} should hold true for the parameter 
triplets marked with an Asterisk $^*$ in the 
scheme of Example~\ref{ex:42}.

\begin{proposition}\label{prop:73}
For parameter triplets $(\vartheta,\tau,\sigma)$ satisfying both Conjectures~\ref{conjecture:A} 
and \ref{conjecture:B}, 
the event 
\[
R(\infty) :=\liminf\limits_{T_1\in\mathbbm{N} , T_1\to\infty} R(T_1) 
\]
in $\mathcal{G}_\infty = \mathcal{C}$ is of full measure under 
$Q_\mu^{(\vartheta,\tau,\sigma)}$.
\end{proposition}
\begin{proof}
	Accept Conjectures~\ref{conjecture:A} and \ref{conjecture:B} for parameters 
	$(\vartheta,\tau,\sigma)$ under consideration. 
	
	\begin{enumerate}[1)]
	\item Since by Proposition~\ref{prop:22}~\ref{item:22b})  the empirical distribution 
	functions  
	$\widehat{H}_n$ 
	associated to interspike times 
	\[
	\left(  \tau_2-\tau_1  ,  \ldots  , \tau_{n+1}-\tau_n  \right)   \quad,\quad n\to\infty
	\]
	converge almost surely under $(\vartheta,\tau,\sigma)$, uniformly on $[0,\infty)$, to a limit 
	distribution function  $H^{(\vartheta,\tau,\sigma)}$, continuity and  strict  monotonicity 
	of the limit --stated in Conjecture~\ref{conjecture:A}-- imply almost sure convergence of 
	$\alpha$-quantiles, 
	$0<\alpha<1$. The same assertion then holds for empirical distribution functions associated 
	to interspike times  \eqref{regular_ISI} observed up to time $T_1$ 
	\[
	\left(  \tau_2-\tau_1  ,  \ldots  , \tau_{N_{T_1}}-\tau_{N_{T_1}-1}  \right)  \quad,\quad 
	T_1\to\infty 
	\]
	with the same limit distribution function $H^{(\vartheta,\tau,\sigma)}$. 
	As a consequence, we have almost sure convergence of $\alpha$-quantiles in data sets 
	\eqref{regular_ISI} to those of $H^{(\vartheta,\tau,\sigma)}$. In particular, with notations   
	\eqref{regular_ISI}--\eqref{ratio_regular_ISI}, 
	\begin{equation}\label{crucial_convergences_1}
		\Delta(T_1)\longrightarrow \Delta^{(\vartheta,\tau,\sigma)} \quad,\quad 
		\mathtt{r}(\alpha,T_1)\longrightarrow \mathtt{r}^{(\vartheta,\tau,\sigma)} (\alpha)  
	\end{equation}
	converge almost surely as  $T_1\to\infty$; in the limit appear the corresponding quantities 
	defined from the limit distribution $H^{(\vartheta,\tau,\sigma)}$. 
	\item 
	Accepting Conjecture~\ref{conjecture:B}, we have  
	\[
	\mathtt{r}^{(\vartheta,\tau,\sigma)}(0.05)  < 0.3 \quad,\quad 
	\mathtt{r}^{(\vartheta,\tau,\sigma)}(0.1)  < 0.2 \quad,\quad 
	\mathtt{r}^{(\vartheta,\tau,\sigma)}(0.25)  < 0.1 . 
	\]
	In virtue of \eqref{crucial_convergences_1}, as $T_1\in\mathbbm{N}$ and $T_1\to\infty$, we 
	do have 
	\begin{equation}\label{crucial_convergences_2}
		\mathtt{r}(0.05,T_1) \le 0.3  \quad,\quad 
		\mathtt{r}(0.1,T_1) \le 0.2   \quad,\quad 
		\mathtt{r}(0.25,T_1) \le 0.1  
	\end{equation}
	for eventually all $T_1\in\mathbbm{N}$. Similarly, accepting Conjecture~\ref{conjecture:B} 
	we have 
	\[
	\left| E_\mu^{(\vartheta,\tau,\sigma)}\left( N_1 \right) \Delta^{(\vartheta,\tau,\sigma)}  -  
	1  \right| < 0.05  
	\]
	for the limit distribution  $H^{(\vartheta,\tau,\sigma)}$, and thus, combining 
	\eqref{crucial_convergences_1} with Proposition~\ref{prop:24}, 
	\begin{equation}\label{crucial_convergences_3} 
		\left| \frac{ N_{T_1} \Delta(T_1) }{ T_1 }  -  1  \right|\le 0.05  
	\end{equation}
	for eventually all $T_1\in\mathbbm{N}$. By definition of the events 
	$R(T_1)\in\mathcal{G}_{T_1}$ in Definition~\ref{def:41} and by 
	Proposition~\ref{prop:22}~\ref{item:22a}) , the 
	assertion is proved.
\end{enumerate}
\end{proof}

\begin{remark}\label{rem:74}
	By \eqref{crucial_convergences_1}, Conjecture~\ref{conjecture:A} implies 
that for all 
$(\vartheta,\tau,\sigma)$, the random variables  \eqref{benchmarks_T1}
\[
\sum_{j\ge 0} e^{ - c_1 \Delta(T_1) j } = \frac{ 1 }{ 1 -  e^{ - c_1 \Delta(T_1) } } 
\quad,\quad 
\sum_{j\ge 0} e^{ - c_1 \Delta(T_1) (j+1) } = \frac{ e^{ - c_1 \Delta(T_1) } }{ 1 -  e^{ - c_1 
		\Delta(T_1) } } 
\]
converge almost surely as $T_1\to\infty$ to the deterministic limits 
\begin{equation}\label{benchmarks_limit}
	\sum_{j\ge 0} e^{ - c_1 \Delta^{(\vartheta,\tau,\sigma)} j } = \frac{ 1 }{ 1 -  e^{ - c_1 
			\Delta^{(\vartheta,\tau,\sigma)}  } } 
	\quad,\quad 
	\sum_{j\ge 0} e^{ - c_1 \Delta^{(\vartheta,\tau,\sigma)}  (j+1) } = \frac{ e^{ - c_1 
			\Delta^{(\vartheta,\tau,\sigma)}  } }{ 1 -  e^{ - c_1 \Delta^{(\vartheta,\tau,\sigma)} } } 
\end{equation}
defined in terms of the median of the limit distribution $H^{(\vartheta,\tau,\sigma)}$. 
\end{remark}

\begin{remark}\label{rem:75}
	Assume that $(\vartheta,\tau,\sigma)$ satisfies Conjectures~\ref{conjecture:A} and 
	\ref{conjecture:B}.    Then, for every 
	$L\in\mathbbm{N}$ fixed, the limit distribution function $G_\mu^{(\vartheta,\tau,\sigma)}$ 
	on 
	$[0,\infty)^L$ from Theorem~\ref{theo:25}, almost sure limit of empirical distribution 
	functions 
	$\widehat{G}_m$ associated to the first $m$ $L$-tuples   out of 
	\[
	\left( \tau_{n+1}-\tau_n , \ldots ,  \tau_{n+L}-\tau_{n+L-1} \right) \quad,\quad 
	n\in\mathbbm{N} 
	\]
	as $m\to\infty$, is concentrated on neighbourhoods of the point 
\[	
\left( \Delta^{(\vartheta,\tau,\sigma)} , \ldots , \Delta^{(\vartheta,\tau,\sigma)}  \right) 
	\in [0,\infty)^L
	\]
	in the sense that marginals, i.e.\ image measures under projection on single coordinates 
	$i\in\{1,\ldots,L\}$, 
	admit $\Delta^{(\vartheta,\tau,\sigma)}$ as their median and bounds 
	\[
	\mathtt{r}^{(\vartheta,\tau,\sigma)}(0.05) < 0.3, 
	\mathtt{r}^{(\vartheta,\tau,\sigma)}(0.1) < 0.2, 
	\mathtt{r}^{(\vartheta,\tau,\sigma)}(0.25) < 0.1  
	\]
	on ratios of distances between upper and lower quantiles divided by the median. 
	To see this, it is sufficient to note that necessarily every marginal of the law 
	$G_\mu^{(\vartheta,\tau,\sigma)}$ in Theorem~\ref{theo:25} coincides with the probability 
	measure 
	$H^{(\vartheta,\tau,\sigma)}$ of Proposition~\ref{prop:22}~\ref{item:22b}) .
\end{remark}

\begin{remark}\label{rem:76}
Grant Conjectures~\ref{conjecture:A} and \ref{conjecture:B} for 
$(\vartheta,\tau,\sigma)$.    Then 
\begin{enumerate}[a)]
\item\label{item:rem76a} expressions  \eqref{benchmarks_limit} provide deterministic 
benchmarks for the location of 
 $(U_{\tau_\ell} , U_{\tau_{ (\ell+1)^-} })$ to be observed  under 
$(\vartheta,\tau,\sigma)$ in the long run as $\ell\to\infty$,  
\item\label{item:rem76b} expressions  \eqref{benchmarks_T1} provide 
$\mathcal{G}_{T_1}$-measurable 
approximations to  \eqref{benchmarks_limit}, converging to \eqref{benchmarks_limit} as 
$T_1\to\infty$. 
\end{enumerate}
\end{remark}

\begin{proof} 
	\begin{enumerate}[i)] 
		\item Note first that the sequence of local maxima $(U_{\tau_\ell})_\ell$ in the output 
		process 
	$U$ is bounded:   
	interspike times being bounded away from $0$ by definition in \eqref{def_spiketimes}, so 
	local maxima of $U$ --using \eqref{output_properties}-- take values in the compact set 
	$K:=\left[ 1 ,\sum_{j\ge 0} e^{- c_1 j \delta_0 } \right]$. 
	\item\label{item:pointiipage36}  For arbitrary $\ell,n,m \in\mathbbm{N}$ we have 
	$U_{\tau_\ell} = 
	\sum\limits_{j=1}^\ell 
	e^{- c_1 ( \tau_\ell-\tau_j ) } $ and thus 
	\begin{equation}\label{U_zerlegung}
		U_{\tau_{n+m}} = U_{\tau_n} e^{ - c_1 (\tau_{n+m} - \tau_n) }  + 
		\sum_{j=1}^m e^{- c_1 ( \tau_{n+m}-\tau_{n+j} ) }. 
	\end{equation}
	If for $L$ large enough an $L$-tuple of interspike times as considered in 
	Remark~\ref{rem:75}  
	\[
	\left( \tau_{n+1}-\tau_n , \ldots   \tau_{n+L}-\tau_{n+L-1} \right) \quad,\quad 
	n\in\mathbbm{N} 
	\]
	is well concentrated at $\Delta^{(\vartheta,\tau,\sigma)}$, then values of 
	$U_{\tau_{n+L}}$ will be close to 
	\[
	u_\infty^{(\vartheta,\tau,\sigma)} := \sum_{j\ge 0} e^{- c_1 j \Delta^{(\vartheta,\tau,\sigma)} 
	} 
	= \frac{1}{ 1 - e^{ - c_1 \Delta^{(\vartheta,\tau,\sigma)}} }
	\]
	no matter where $U_{\tau_n}$ was located in $K$. Hence small neighbourhoods of 
	$u_\infty^{(\vartheta,\tau,\sigma)}$ are  attainable for the process  $(U_{\tau_\ell})_\ell$ of 
	local maxima of the output process. 
	\item Whenever   $U_{\tau_n}$ in \eqref{U_zerlegung} is close to  
	$u_\infty^{(\vartheta,\tau,\sigma)}$  for some $n$, an $L$-tuple of interspike times well 
	concentrated at $\Delta^{(\vartheta,\tau,\sigma)}$ as in \ref{item:pointiipage36})  allows to 
	write in good  
	approximation 
	\[
	U_{\tau_{n+L}}\approx u_\infty^{(\vartheta,\tau,\sigma)} e^{- c_1 L 
	\Delta^{(\vartheta,\tau,\sigma)} }  + 
	\sum_{j=1}^L e^{- c_1 ( L-j ) \Delta^{(\vartheta,\tau,\sigma)} } = 
	u_\infty^{(\vartheta,\tau,\sigma)}. 
	\]
	Hence also $U_{\tau_{n+L}}$ will be close to $u_\infty^{(\vartheta,\tau,\sigma)}$. This 
	shows that  small neighbourhoods of $u_\infty^{(\vartheta,\tau,\sigma)}$ will be attained 
	infinitely often by the process of local maxima in the long run. 
	\item  Thus asymptotically as $\ell\to\infty$,  
	pairs $\left( U_{\tau_\ell} , U_{ ( \tau_{\ell+1} )^- }\right)$ will visit small 
	neighbourhoods of 
	\[
	\left( u_\infty^{(\vartheta,\tau,\sigma)} , u_\infty^{(\vartheta,\tau,\sigma)} e^{ - c_1 
		\Delta^{(\vartheta,\tau,\sigma)} }  \right)= \left( \frac{1}{ 1 - e^{ - c_1 
			\Delta^{(\vartheta,\tau,\sigma)}} } ,\frac{ e^{ - c_1 \Delta^{(\vartheta,\tau,\sigma)} } }{ 
			1 - e^{ 
			- c_1 \Delta^{(\vartheta,\tau,\sigma)}} } \right)
	\]
	infinitely often. In this sense, the deterministic expression \eqref{benchmarks_limit} 
	provides a benchmark for the location of pairs $\left( U_{\tau_\ell} , U_{ ( \tau_{\ell+1} 
	)^- }\right)$ under $(\vartheta,\tau,\sigma)$ in the long run as $\ell\to\infty$, in virtue of 
	our two conjectures. This is \ref{item:rem76a}). 
	By \eqref{crucial_convergences_1}, $\Delta(T_1)$ converges to 
	$\Delta^{(\vartheta,\tau,\sigma)} $ almost surely as $T_1\to\infty$. So if we have observed 
	the stochastic neuron up to time $T_1$, for $T_1$ large enough, expressions 
	\eqref{benchmarks_T1}
	\[
	\left( \frac{1}{ 1 - e^{ - c_1 \Delta(T_1) } } ,\frac{ e^{ - c_1 \Delta(T_1) } }{ 1 - e^{ - c_1 
	\Delta(T_1) } } \right)
	\]
	are $\mathcal{G}_{T_1}$-measurable approximations to the benchmark in 
	\eqref{benchmarks_limit}. This is \ref{item:rem76b}). 
\end{enumerate}
\end{proof}



\providecommand{\bysame}{\leavevmode\hbox to3em{\hrulefill}\thinspace}



\nocite{*}
\end{document}